\documentclass[a4paper,11pt]{article}

\usepackage{amsmath,amsthm,amssymb,mathrsfs,latexsym,amsfonts}
\usepackage{graphicx,psfrag,epsfig}
\usepackage[english]{babel}
\usepackage[latin1]{inputenc}

\newtheorem{theorem}{Theorem}[section]
\newtheorem{lemma}[theorem]{Lemma}
\newtheorem{proposition}[theorem]{Proposition}

\newtheorem{corollary}[theorem]{Corollary}
\newtheorem{definition}[theorem]{Definition\rm}
\newtheorem{remark}[theorem]{\it Remark}

\newcounter{paraga}[section]
\renewcommand{\theparaga}{{\bf\arabic{paraga}.}}
\newcommand{\paraga}{\medskip \addtocounter{paraga}{1} 
\noindent{\theparaga\ } }

\begin{document}

\bibliographystyle{amsalpha}

\def\MP{\,{<\hspace{-.5em}\cdot}\,}
\def\SP{\,{>\hspace{-.3em}\cdot}\,}
\def\PM{\,{\cdot\hspace{-.3em}<}\,}
\def\PS{\,{\cdot\hspace{-.3em}>}\,}
\def\EP{\,{=\hspace{-.2em}\cdot}\,}
\def\PP{\,{+\hspace{-.1em}\cdot}\,}
\def\PE{\,{\cdot\hspace{-.2em}=}\,}
\def\N{\mathbb N}
\def\C{\mathbb C}
\def\Q{\mathbb Q}
\def\R{\mathbb R}
\def\T{\mathbb T}
\def\A{\mathbb A}
\def\Z{\mathbb Z}
\def\demi{\frac{1}{2}}

\begin{titlepage}
\author{Abed Bounemoura~\footnote{Abed.Bounemoura@math.u-psud.fr, Laboratoire Math\'ematiques d'Orsay \& Institut Math\'ematiques de Jussieu} {} and Laurent Niederman~\footnote{Laurent.Niederman@math.u-psud.fr, Laboratoire Math\'ematiques d'Orsay \& IMCCE, Observatoire de Paris}}
\title{\LARGE{\textbf{Generic Nekhoroshev theory without small divisors.}}}
\end{titlepage}

\maketitle

\begin{center}
{\it Dedicated to the memory of N.N. Nekhoroshev (1946-2008)}
\end{center}

\begin{abstract}
In this article, we present a new approach of Nekhoroshev's theory for a generic unperturbed Hamiltonian which completely avoids small divisors problems. The proof is an extension of a method introduced by P. Lochak, it combines averaging along periodic orbits with simultaneous Diophantine approximation and uses geometric arguments designed by the second author to handle generic integrable Hamiltonians. This method allows to deal with generic non-analytic Hamiltonians and to obtain new results of generic stability around linearly stable tori.   
\end{abstract}
  
\section{Introduction}

\paraga In this article, we are concerned with the stability properties of near-integrable analytic Hamiltonian systems. According to a classical theorem of Liouville-Arnold (see \cite{AKN97}), such systems are locally governed by a Hamiltonian of the form
\begin{equation*}
\begin{cases} 
H(\theta,I)=h(I)+f(\theta,I) \\
|f| < \varepsilon  <\!\!<1
\end{cases}
\end{equation*}
where $(\theta,I) \in \T^n \times \R^n$ are action-angle coordinates for $h$ and $f$ is a small perturbation in some suitable topology. For the integrable system, that is when $f=0$, the action variables of solutions are trivially constant for all times, but when $f \neq 0$ they are no longer constant of motions and we are interested in studying their evolution for long intervals of time. 

\paraga But first it is important to understand the integrable case. When $H=h$ depends only on the action variables, as the latter are constant for all times, the phase space is trivially foliated into invariant tori $\mathcal{T}_{I_0}=\T^n \times \{I_0\}$, for $I_0 \in \R^n$, and on each torus $\mathcal{T}_{I_0}$ the flow is quasi-periodic with frequency vector $\omega_0=\nabla h(I_0) \in \R^n$. The dynamics of such a flow is completely understood and depends on the frequency vector $\omega_0$, more precisely on its resonant module
\[ \mathcal{M}(\omega_0)=\{k \in \Z^n \; | \; k.\omega_0=0\},\] 
where the dot denotes the Euclidean scalar product. If $\mathcal{M}(\omega_0)$ is trivial, then the dynamics is minimal and uniquely ergodic. Otherwise, we have a relation of the form $k.\omega_0=0$ for some $k \in \Z^n \setminus \{0\}$, which is usually called a resonance, and denoting by $m$ the rank of $\mathcal{M}(\omega_0)$, the torus $\mathcal{T}_0$ splits into a continuous $m$-parameter family of invariant sub-tori of dimension $n-m$, on which the dynamics is minimal and uniquely ergodic. These are called resonant tori, and in case of maximal resonances ({\it i.e.} $m=n-1$ if $h$ does not have critical points), the tori are foliated into periodic orbits. Under some non-degeneracy assumption on $h$, both resonant and non-resonant tori form a dense subset of the phase space.

\paraga Returning to the perturbed system, since Poincar\'e we know that resonant tori do not survive (actually he proved that for a periodic tori, generically only a finite number of periodic orbits persist). But it was a remarkable idea of Kolmogorov (\cite{Kol54}) to focus on non-resonant tori to prove that a set of large measure of invariant tori survives under some regularity and non-degeneracy assumptions. This has now become a rich and vast subject called KAM theory (see \cite{Pos01}, \cite{dlL01} or \cite{Bos86} for some nice introductions on this theory). Such tori persist in a $\sqrt\varepsilon$-neighbourhood of the unperturbed ones and therefore for a set of large measure of initial conditions, the variation of the actions is of order $\sqrt\varepsilon$ for {\it all time}. But on the other hand, this set of KAM tori is typically a Cantor family (hence with no interior) and the theory gives no information on the complement, except when $n=2$ where these two-dimensional invariant tori disconnect the three-dimensional energy level leaving {\it all solutions} stable for {\it all time}. However for $n\geq 3$, it is still possible to find solutions for which the variation of the action components is of order one. A proof of this fact was outlined by Arnold in his famous paper (\cite{Arn64}) where he proposed a mechanism to produce examples of near-integrable Hamiltonian systems where such a drift occurs no matter how small the perturbation is. This phenomenon is usually referred to Arnold diffusion.

\paraga Hence for $n\geq 3$, results of stability for near-integrable Hamiltonian systems which are valid for an open set of initial conditions can only be proved over finite times. This picture was completed by Nekhoroshev in the seventies (see \cite{Nek77},\cite{Nek79} and \cite{Nie09} for a recent overview of the theory) who proved the following: if the system is analytic and the unperturbed Hamiltonian $h$ satisfies some quantitative transversality condition called {\it steepness}, then there exist positive constants $a$, $b$, $\varepsilon_0$, $c_1$, $c_2$ and $c_3$ depending only on $h$, such that every solution $(\theta(t),I(t))$
of the perturbed system starting at time $t=0$ satisfies
\begin{equation}\label{eq2}
|I(t)-I(0)| \leq c_1 \varepsilon^b, \quad |t| \leq c_2\exp\left(c_3 \varepsilon^{-a}\right), 
\end{equation}
provided that the size of the perturbation $\varepsilon$ is smaller than the threshold $\varepsilon_0$. The constants $a$ and $b$ are called the stability exponents. If property (\ref{eq2}) is satisfied, we shall say that the integrable Hamiltonian $h$ is exponentially stable. Hence, KAM and Nekhoroshev's theory yield different type of stability results, but they both ultimately rely on the same tool which is the construction of normal forms, and we shall described it below.

\paraga The basic idea is  to look at a ``more integrable" Hamiltonian which yields a good approximation of the perturbed system. By the averaging principle (see \cite{AKN97}), this simpler Hamiltonian is given by the time average of the system along the unperturbed flow, that is
\[ [H]=h+[f],\]
where 
\[ [f]=\lim_{t \rightarrow \infty} \left( \frac{1}{t}\int_{0}^{t} f\circ\Phi_{s}^{h}ds\right ), \]
and $\Phi_{s}^{h}$ is the Hamiltonian flow of the integrable part $h$. Actually, this average depends on the dynamics of the unperturbed Hamiltonian and hence on resonant modules associated to frequencies. So given a sub-module $\mathcal{M}\subseteq\Z^n$, we define its resonant manifold by
\begin{equation*}
S_{\mathcal{M}}=\left\{ I\in\R^n \; | \;  k.\nabla h(I)=0\ {\rm for}\ k\in{\mathcal M}\right\}.
\end{equation*}
Due to the ergodic properties of the linear flow with vector $\nabla h(I)$ over the torus $\T^n$, the time average over $S_{\mathcal{M}}$ equals the space average along a torus of dimension $n-m$ if $m$ is the multiplicity of the resonance ({\it i.e.} the rank of $\mathcal{M}$), hence $n-m$ angles have been removed in this case. From a physical point of view, the guiding principle is that rapidly oscillating terms discarded in averaging cause only small oscillations which are superimposed to the solutions of the averaged system. In order to prove this claim, one should check that any solution of the perturbed system remains close to the solution of the averaged system with the same initial condition. Especially, this will be the case if one finds a canonical transformation $\varepsilon$-close to identity which conjugates the perturbed Hamiltonian to its average. Hence we are reduced to a problem of {\it normal form} where one tries to conjugate the system to a simpler one, that is we look for a convenient system of coordinates.

However, constructing such a good system of coordinates is not an easy task. The linearised equation of conjugation reads
\begin{equation*}
\{\chi,h\}=f-[f],
\end{equation*}
if $\chi$ is the function generating the conjugation. This is usually called a {\it homological equation} and to solve it we need to invert the linear operator $L_h=\{.,h\}$ acting on a suitable space of functions. Here our operator is invertible, but its inverse is generally unbounded: this is the {\it small divisors} phenomenon. To see this, just note that once an action $I \in S_{\mathcal{M}}$ is fixed (and hence a frequency $\omega=\nabla h(I)$ satisfying $k.\omega \neq 0$ for $k\notin\mathcal{M}$), the homological equation is a just a first-order, linear with constant coefficients partial differential equation on $\T^n$, namely
\[ \omega.\nabla \chi=f-[f]. \]
Such equations are known to be well-suited for Fourier analysis, in our case the operator $L_h$ is easily diagonalized in a Fourier basis and we find that the eigenvalues are proportional to the scalar products $k.\omega$, for $k \in \Z^n$. More precisely, expanding $\chi$ and $f$ as 
\[ \chi(\theta)=\sum_{k \in \Z^n}\hat{\chi}_k e^{i2\pi k.\theta}, \quad f(\theta)=\sum_{k \in \Z^n}\hat{f}_k e^{i2\pi k.\theta}, \]
then
\[ [f]=\sum_{k \in \mathcal{M}}\hat{f}_k e^{i2\pi k.\theta}, \]
and so formally
\begin{equation} \label{diviseurs}
\hat{\chi}_k=
\begin{cases}
\left(i2\pi k.\omega\right)^{-1}\hat{f}_k, \; k \notin \mathcal{M}, \\
0 , \; k \in \mathcal{M}.
\end{cases}
\end{equation}
The scalar products $k.\omega$ appearing in the denominators of~(\ref{diviseurs}) are not zero by assumption, but they can be arbitrarily small and this is inevitable for large integers $k$ (see the estimate~(\ref{smalld}) below). This can cause the divergence of the Fourier series of $\chi$ and hence the unboundedness of the inverse of $L_h$. Classical small divisors techniques are concerned with obtaining lower bounds for these scalar products to ensure the convergence of the series and this leads necessarily to complicated estimates. Furthermore, to obtain a result applying to all solutions, a partition of the phase space into resonant manifolds associated to different modules, usually called the {\it geometry of resonances}, has to be achieved and this is a delicate task. All these techniques are very important, in particular to study Arnold diffusion and related problems, however we will show that they are not necessary to prove Nekhoroshev's estimates.

\paraga Indeed, all these problems are completely bypassed if we only average along periodic orbits of the unperturbed flow. We first recall the following definition. 

\begin{definition}
A vector $\omega \in \R^n$ is said to be periodic if there exists a real number $t>0$ such that $t\omega \in \Z^n$. In this case, the number 
\[ T=\inf\{t>0 \; | \; t\omega \in \Z^n\}\] 
is called the period of~$\omega$.
\end{definition}

A basic example is given by a vector with rational components, the period of which is just the least common multiple of the denominators of its components. Geometrically, if $\omega$ is $T$-periodic, an invariant torus with a linear flow with vector $\omega$ is filled with $T$-periodic orbits. In this case, the average along such a periodic solution is given by
\[[f]=\lim_{t \rightarrow \infty} \left(\frac{1}{t}\int_{0}^{t} f\circ\Phi_{s}^{l}ds\right)=\frac{1}{T}\int_{0}^{T}f\circ\Phi_{s}^{l}ds, \]
where $l$ denotes the linear Hamiltonian with frequency $\omega$, that is $l(I)=\omega.I$. Then the homological equation $\{\chi,l\}=f-[f]$ is easily solved without using Fourier expansions and is given by an explicit integral formula
\[ \chi=\frac{1}{T}\int_{0}^{T}(f-[f])\circ\Phi_{s}^{l}sds. \]
So in this case, there is no small divisors. To understand more concretely the previous sentence, consider a vector $\omega \in \R^n$ and multi-integers $k$ that do not resonate with $\omega$ (that is $k \notin \Z^n \cap \omega^{\perp}$). Then in general we don't have a lower bound on the divisors $k.\omega$ that appears in~(\ref{diviseurs}), and by a theorem of Dirichlet one has the upper bound
\begin{equation}\label{smalld}
\min_{0<|k| \leq K}|k.\omega| \leq \frac{|\omega|}{K^{n-1}}. 
\end{equation}
In that context, small divisors techniques use Diophantine vectors for which $|k.\omega| \geq \gamma |k|_{1}^{-\tau}$, with $\gamma>0$, $\tau\geq n-1$ and where $|\,.\,|_1$ stands for the $\ell^1$-norm, but nevertheless the lower bound deteriorates as $|k|_1$ increases, causing extra difficulties (which are usually handled by the so-called ultra-violet cut-off). However if the vector $\omega$ is $T$-periodic, one simply has $|k.\omega| \geq T^{-1}$ and the lower bound is uniform in $|k|_1$.

\paraga Lochak (\cite{Loc92}, see also \cite{LN92} and \cite{LNN94} for refinements) has shown that averaging along the periodic orbits of the integrable Hamiltonian is enough to obtain Nekhoroshev's estimates of stability when the unperturbed Hamiltonian is strictly {\it convex} (or strictly quasi-convex, that is the Hamiltonian is strictly convex when restricted to its energy sub-levels). Indeed, using convexity, Lochak obtains open sets around periodic orbits over which exponential stability holds. Then, Dirichlet's theorem about {\it simultaneous Diophantine approximation} ensures easily that these open sets recover the whole action space and yields the global result, avoiding the difficult geometry of resonances. Put it differently, in the convex case one only needs dynamical informations near resonances of maximal multiplicities, which are completely characterized by periodic orbits.

The goal of this paper is to extend Lochak's approach for a generic set of integrable Hamiltonians. To do so, we will have to analyze the dynamics in a neighbourhood of suitable resonances of any multiplicities by using only successive averagings along periodic orbits together with Dirichlet's theorem, and this will lead to exponential estimates of stability for perturbation of a generic integrable Hamiltonian, as stated below.

\begin{theorem} \label{mainth}
Consider an arbitrary real analytic integrable Hamiltonian $h$ defined on a neighbourhood of a closed ball in $\R^n$. Then for almost any $\xi\in\R^n$, the integrable Hamiltonian $h_\xi(x)=h(I)-\xi.I$ is exponentially stable with the exponents $a=b=3^{-1}(2n)^{-3n}$.
\end{theorem}

This will be a direct consequence of Theorems~\ref{thmpreva} and~\ref{mainth2}, see below in section~\ref{s21}. This result is not new, see \cite{Nie07}, but the novelty here is our method of proof, which avoids completely the fundamental problem of small divisors and hence all the associated technicalities (non-resonant domains, Fourier series, Fourier norm, ultra-violet cut-off and so on). The analytic part of our proof of Nekhoroshev's estimates is therefore reduced to its bare minimum, it is nothing but a classical one-phase averaging, while our geometric part is based on a clever use of Dirichlet's theorem along each solution. Applications of our method to other problems will be discussed below, in section~\ref{s22}.

\bigskip

To conclude this introduction, we point out that the method of averaging along periodic orbits has also been used successfully to re-prove recently some KAM theorems without small divisors (see \cite{KDM06} and \cite{KDM07}), even though their techniques are much more complicated.  

\section{Statement of results} \label{s2}

\subsection{Set-up and results} \label{s21}

\paraga Let $B=B_R$ be the open ball centered at the origin of $\R^n$ of radius $R$ with respect to the supremum norm, the domain $\mathcal{D}=\T^n \times B$ will be our phase space. To avoid trivial situations, we assume $n\geq2$. Our Hamiltonian function $H$ is real-analytic and bounded on $\mathcal{D}$ and it admits a holomorphic extension to some complex neighbourhood of $\mathcal{D}$ of the form
\[ \mathcal{D}_{r,s}=\{(\theta,I)\in(\C^n/\Z^n)\times \C^{n} \; | \; |\mathcal{I}(\theta)|<s,\;d(I,B)< r\}, \] 
with two fixed numbers $r>0$, $s>0$, and where $\mathcal{I}(\theta)$ is the imaginary part of $\theta$, $|\,.\,|$ the supremum norm on $\C^n$ and $d$ the associated distance on $\C^n$. Equivalently, one can start with a Hamiltonian $H$, defined and holomorphic on $\mathcal{D}_{r,s}$ and which preserves reality, that is $H$ is real-valued for real arguments. Without loss of generality, we may assume that $r<1$ and $s<1$. The space of such analytic functions on $\mathcal{D}_{r,s}$, equipped with the supremum norm $|\,.\,|_{r,s}$, is obviously a Banach algebra with respect to the multiplication of functions, and we shall denote it by $\mathcal{A}_{r,s}$. 

Our Hamiltonian $H \in \mathcal{A}_{r,s}$ is assumed to be close to integrable, that is of the form
\begin{equation}\label{Ham}
\begin{cases} \tag{$\ast$}
H(\theta,I)=h(I)+f(\theta,I) \\
|f|_{r,s}<\varepsilon  <\!\!< 1,
\end{cases}
\end{equation}
where $h$ is the integrable part and $f$ a small perturbation. Moreover, the derivatives up to order 3 of $h$ are assumed to be bounded by some constant $M>1$, that is
\begin{equation*}
|\partial ^k h(I)|\leq M, \quad 1\leq |k|_1\leq 3, \quad I\in B,
\end{equation*}
where $|k|_1=|k_1|+\cdots+|k_n|$.

\paraga In order to obtain results of exponential stability, we do need to impose some non-degeneracy condition on the unperturbed Hamiltonian. Let $G(n,k)$ be the set of all vector subspaces of $\R^n$ of dimension $k$. We equip $\R^n$ with the Euclidean scalar product, $\Vert\,.\,\Vert$ stands for the Euclidean norm, and given an integer $L\in\N^*$, we define $G^{L}(n,k)$ as the subset of $G(n,k)$ consisting of those subspaces whose orthogonal complement can be spanned by vectors $k\in\Z^n$ with $|k|_1\leq L$.   

\begin{definition} \label{sdm}
A function $h\in C^2(B)$ is said to be SDM if there exist $\gamma>0$ and $\tau \geq 0$ such that for any $L\in\N^*$, any $k\in\{1,\dots,n\}$ and any $\Lambda\in G^{L}(n,k)$, there exists $\left(e_1,\dots,e_k\right)$ (resp. $\left(f_1,\ldots,f_{n-k}\right)$), an orthonormal basis of $\Lambda$ (resp. of $\Lambda^\perp$), such that the function $h_\Lambda$ defined on $B$ by 
\[ h_\Lambda(\alpha,\beta)=h\left(\alpha_1 e_1+\dots+\alpha_k e_k+\beta_1f_1+\dots+\beta_{n-k} f_{n-k}\right), \]
satisfies the following: for any $(\alpha,\beta) \in B$, 
\[ \Vert\partial_\alpha h_\Lambda(\alpha,\beta)\Vert \leq \gamma L^{-\tau} \Longrightarrow \Vert \partial_{\alpha\alpha} h_\Lambda(\alpha,\beta).\eta\Vert>\gamma L^{-\tau}\Vert\eta\Vert \]
for any $\eta \in \R^n\setminus\{0\}$.
\end{definition} 

In other words, for any $(\alpha,\beta) \in B$, we have the following alternative: either $\Vert \partial_\alpha h_\Lambda(\alpha,\beta)\Vert>\gamma L^{-\tau}$ or $\Vert\partial_{\alpha\alpha} h_\Lambda(\alpha,\beta).\eta\Vert>\gamma L^{-\tau}\Vert\eta\Vert$ for any $\eta \in \R^n\setminus\{0\}$. This technical definition, which is a slight variation of a notion introduced in \cite{Nie07}, is basically a quantitative transversality condition which is stated in adapted coordinates. It is inspired on the one hand by the steepness condition introduced by Nekhoroshev (\cite{Nek77}) where one has to look at the projection of the gradient map $\nabla h$ onto affine subspaces, and on the other hand by the quantitative Morse-Sard theory of Yomdin (\cite{Yom83}, \cite{YC04}) where critical or ``nearly-critical" points of $h$ have to be quantitatively non degenerate. The abbreviation SDM stands for ``Simultaneous Diophantine Morse" functions, and we refer to Appendix~\ref{SDM} for more explanations on this condition and some justifications on the latter terminology.

\paraga The set of SDM functions on $B$ with respect to $\gamma>0$ and $\tau \geq 0$ will be denoted by $SDM_{\gamma}^{\tau}(B)$, and we will also use the notations 
\[ SDM^{\tau}(B)=\bigcup_{\gamma>0}SDM_{\gamma}^{\tau}(B), \quad  SDM(B)=\bigcup_{\tau \geq 0}SDM^{\tau}(B). \]
The following result states that SDM functions are generic among sufficiently smooth functions. 

\begin{theorem}\label{thmpreva}
Let $\tau>2(n^2+1)$ and $h\in C^{2n+2}(B)$. Then for Lebesgue almost all $\xi \in \R^n$, the function $h_\xi(I)=h(I)-\xi.I$ belongs to $SDM^{\tau}(B)$.  
\end{theorem}

More precisely, there is a good notion of ``full measure" in an infinite dimensional vector space, which is called prevalence (see \cite{OY05} and \cite{HK10} for nice surveys), and the previous theorem immediately gives the following result.

\begin{corollary}\label{corpreva}
For $\tau>2(n^2+1)$, $SDM^{\tau}(B)$ is prevalent in $C^{2n+2}(B)$.  
\end{corollary}

\paraga Now we can state the main result of the paper.  
  
\begin{theorem} \label{mainth2}
Let $H$ as in (\ref{Ham}) and assume that the integrable part $h$ belongs to $SDM_{\gamma}^{\tau}(B)$ with  $\tau\geq 2$ and $\gamma\leq 1$. Then there exist positive constants $a$ and $b$ depending only on $n$ and $\tau$, and $\varepsilon_0$ depending only on $h$, such that if $\varepsilon \leq \varepsilon_0$, for every initial action $I(0) \in B_{R/2}$ the following estimates
\[ |I(t)-I(0)| < (n+1)^2\varepsilon^b, \quad |t| < \exp(\varepsilon^{-a}), \]
hold true.
\end{theorem}

More precisely, we can choose the exponents
\[ a=b=3^{-1}(2(n+1)\tau )^{-n},\] 
and $\varepsilon_0$ depending on the whole set of parameters $n,R,r,s,M,\gamma$ and $\tau$, but no efforts was made to improved the stability exponents since the optimality of the constants involved is not our goal. Actually, this optimality is not relevant for generic integrable Hamiltonians. 

Let us add that the only property used on the integrable part $h$ to derive these estimates is a specific steepness property, therefore the proof is also valid, and in fact simpler, assuming the original steepness condition of Nekhoroshev (see Appendix~\ref{SDM}). However, note that this is precisely this ``weaker" genericity assumption that allows new results of stability near linearly stable invariant tori (see \cite{Bou09}).  

We emphasized again that this is not the result itself, but the method of proof which is new and leads to many improvements as we explain below. 

\subsection{Comments and prospects} \label{s22}

To conclude this section we mention other problems for which our method should apply, mainly the study of elliptic fixed points, Nekhoroshev's estimates in lower regularity and finally estimates in large or infinite dimensional Hamiltonian systems. In all these topics, the method of periodic averagings have already proved to be very useful.

\paraga First our analytic arguments are very intrinsic and this is important in the study of the stability of elliptic fixed points in Hamiltonian systems. Actually, in this case the transformation in action-angle variables (via the symplectic polar coordinates) admits singularities which do not allow to derive directly stability results from Nekhoroshev's theory. In the convex case, this problem has been overcomed independently by Fass\`o, Guzzo and Benettin(\cite{BFG98}) and by Niederman (\cite{Nie98}). Both use Cartesian coordinates, the first study uses the classical approach and adapted Fourier expansions while the second one relies on periodic averagings and simultaneous Diophantine approximation. The latter proof was clarified by P\"oschel (\cite{Pos99a}). With our approach, we can remove the convexity hypothesis to have exponential stability around an elliptic fixed point under a generic assumption on the non-linear part. Furthermore, assuming a Diophantine condition on the normal frequency it is well-known since Morbidelli and Giorgilli (\cite{MG95}) that one can even obtain super-exponential stability by combining a sufficiently large number of Birkhoff normalizations with Nekhoroshev's estimates. Here, with our method generic results of super-exponential stability around elliptic fixed points are also available, and similarly around invariant Diophantine Lagrangian tori and even isotropic reducible linearly stable tori. All this results are contained in \cite{Bou09}. 

\paraga Furthermore, one should mention that periodic averagings are well-suited for non-analytic Hamiltonians and our formalism should also carry on in this context. The advantage of periodic averagings is clear already at the linear level when solving the homological equation: if the system is of finite differentiability, then for a Diophantine frequency vector the solution of the homological equation is subjected to a disastrous loss of derivatives (larger than the number of degrees of freedom) and one has to use rather cumbersome Fourier expansions, while for a periodic frequency vector, this loss of derivatives is minimal and one can use a more elegant integral formula. Hence in finite differentiability, for a convex or generic unperturbed Hamiltonian system, we can expect a proof of stability estimates (with of course a polynomial bound on the time of stability) which is both simple (no small divisors) and direct (no need to use the result in the analytic case and smoothing techniques, which is the usual approach in KAM theory in finite differentiability). Note that the analyticity of the studied system is only needed for the construction of normal forms up to an exponentially small remainder, but our steepness condition is generic for Hamiltonians of {\it finite} but sufficiently high regularity. Concerning Gevrey regularity, Marco and Sauzin (\cite{MS02}) have already proved exponential estimates of stability in the convex case and for the $C^k$ regularity, polynomial estimates of stability are indeed available (see \cite{Bou10}). Both results use only periodic averagings, so with our method they should also hold for a generic integrable Hamiltonian. It can also be noticed that the analytical properties of the expansions arising in periodic averagings are accurately known (\cite{Nei84},\cite{RS96}).

\paraga Finally, results of stability for large Hamiltonian systems as a model for statistical mechanics have been obtained by Bambusi and Giorgilli (\cite{BG93}) and Bourgain (\cite{Bou04}), and for non-linear evolution PDE seen as an infinite dimensional Hamiltonian system mostly by Bambusi (\cite{Bam99}, \cite{BN02}) and then clarified by P\"oschel (\cite{Pos99b}). All these works use Lochak's approach in the convex case. We believe that our method should allow to remove the convexity assumption in those results to obtain more general statements.

\paraga The paper is organized as follows. In the next section, we state our normal form and explain the main ideas, and then we give the proof of Theorem~\ref{mainth2}. The complete proof of the normal form is deferred to Appendix~\ref{estimates}, and in Appendix~\ref{SDM} we collect the basic properties of SDM functions that we shall need and we prove Theorem~\ref{thmpreva} and Corollary~\ref{corpreva}.  

\paraga In the text, we shall adopt the following notation taken from~\cite{Pos99a}: we will write $u \MP v$ if there exists a constant $C\geq1$ such that $u<Cv$, where $C$ depends only on $n,R,r,s,M$, but not on $\tau$ and on the small parameters $\varepsilon$ and $\gamma$. Similarly, we will use the notations $u \PM v$, $u \EP v$ and $u \PE v$.

We shall use the following norms for vectors $v\in\R^n$ or $v\in \C^n$: $|\,.\,|$ will be the supremum norm, $|\,.\,|_1$ the $\ell^{1}$-norm and $\Vert\,.\,\Vert$ the Euclidean (or Hermitian) norm.

\section{Proof of Theorem~\ref{mainth2}} \label{s3}

In this section, we consider the Hamiltonian (\ref{Ham}), that is
\begin{equation*}
\begin{cases}
H(\theta,I)=h(I)+f(\theta,I) \\
|f|_{r,s}<\varepsilon
\end{cases}
\end{equation*}
with $H\in\mathcal{A}_{r,s}$. As usual, the proof of exponential stability estimates splits into an analytic part and a geometric part. 

The analytic part is contained in section~\ref{ss31}. It consists in the construction normal forms on a neighbourhood of specific resonances, that is suitable coordinates which display the relevant part of the perturbation on such a neighbourhood. Basically, we will reduce the perturbation to a so-called resonant term which is dynamically significant, and a general term which will only cause exponentially small deviations.         

The geometric part is expanded in section~\ref{ss32}, and it is mainly based on the properties of the underlying integrable system. The strategy will be first to defined a class of solutions, which we call {\it restrained}, and for which it is obvious from our normal forms that they are stable for an exponentially long time. Using this intermediate result, we will then show that all solutions are in fact exponentially stable, and our main tools to do this will be an adapted steepness property satisfied by our integrable system, as well as a basic theorem of Dirichlet on simultaneous Diophantine approximation. 
    
\subsection{Analytical part} \label{ss31}

\paraga Let us begin by describing the neighbourhoods of resonances we will consider. Given a sequence of linearly independent periodic vectors $(\omega_1,\dots,\omega_n)$, with periods $(T_1,\dots,T_n)$, we define in the complex phase space, for $j\in\{1,\dots,n\}$, the domains
\[ \mathcal{D}_{r_j,s_j}(\omega_j)=\{(\theta,I) \in \mathcal{D}_{r_j,s_j} \; | \; |\nabla h(I)-\omega_j| \MP r_j\}, \] 
with two sequences $(r_1,\dots,r_n)$ and $(s_1,\dots,s_n)$. 

\begin{remark}\label{defdomain}
It is important to note that there is an implicit constant in the previous definition represented by the dot, and we will not make it explicit in order to avoid cumbersome and meaningless expressions. We just mention that it depends only on $n$, $M$ and $j\in\{1,\dots,n\}$ and for subsequent arguments it has to be chosen sufficiently large.
\end{remark}

Informally, one has to view the domain $\mathcal{D}_{r_j,s_j}(\omega_j)$ as a neighbourhood, in frequency space, of a periodic torus with a linear flow of frequency $\omega_j$. Such domains will therefore be called {\it nearly-periodic} tori. We will also use the real part of those domains, which are $\T^n \times \mathcal{B}_{r_j}(\omega_j)$ where 
\[ \mathcal{B}_{r_j}(\omega_j)=\{I\in B_{r_j} \; | \; |\nabla h(I)-\omega_j| \MP r_j\},\]
with $B_{r_j}=\{I\in\R^n\;|\;d(I,B)<r_j\}$.

\paraga Given an analytic function $f$ defined on $\mathcal{D}_{r_j,s_j}(\omega_j)$, we simply denote its supremum norm by 
\[ |f|_{r_j,s_j}=|f|_{\mathcal{D}_{r_j,s_j}(\omega_j)}. \] 
For vector-valued functions, this definition is extended component-wise, that is
\[ |\partial_\theta f|_{r_j,s_j}=\max_{1 \leq i \leq n}|\partial_{\theta_i} f|_{r_j,s_j}, \; |\partial_I f|_{r_j,s_j}=\max_{1 \leq i \leq n}|\partial_{I_i} f|_{r_j,s_j}. \]  

We will write $l_j$ for the linear integrable Hamiltonian with frequency $\omega_j$, that is $l_j(I)=\omega_j.I$ for $j\in\{1,\dots,n\}$. For any function $f$, we will denote $[f]_j$ its average along the periodic flow generated by $l_j$, that is
\[[f]_j=\frac{1}{T_j}\int_{0}^{T_j}f\circ\Phi_{s}^{l_j}ds. \]

\paraga Our interest here is to obtain normal forms on nearly-periodic tori up to an exponentially small remainder with respect to some parameter $m\in\N$, that we will choose later of order $\varepsilon^{-1}$ (during the proof of Theorem~\ref{mainth2}). To this end, we will need the following conditions $(A_j)$, for $j\in\{1,\dots,n\}$, where $(A_1)$ is
\begin{equation} \label{A1}
\begin{cases}
mT_1\varepsilon\PM r_1^2, \; mT_1r_1\PM s_1, \; 0<r_1 \MP s_1, \\
\mathcal{B}_{r_1}(\omega_1)\neq \emptyset, \; r_1\PM r, \; s_1\PM s,  
\end{cases} \tag{$A_1$}
\end{equation} 
and for $j\in\{2,\dots,n\}$, ($A_j)$ is
\begin{equation} \label{Aj}
\begin{cases}
mT_j\varepsilon\PM r_1r_j, \; mT_jr_j\PM s_j, \; 0<r_j \MP s_j,\\
\mathcal{B}_{r_j}(\omega_j)\neq \emptyset,\;\mathcal{D}_{r_j,s_j}(\omega_{j})\subseteq\mathcal{D}_{2r_{j-1}/3,2s_{j-1}/3}(\omega_{j-1}).  
\end{cases} \tag{$A_j$}
\end{equation}
Let us explain briefly our assumptions.

First, the condition on the inclusion of nearly-periodic tori is really crucial. Indeed, since $\omega_1$ is periodic, the nearly-periodic torus $\mathcal{D}_{r_1,s_1}(\omega_1)$ describes a neighbourhood of a resonance of multiplicity $n-1$. Now for $j\in\{2,\dots,n\}$, since $(\omega_1,\dots,\omega_j)$ are periodic and independent, the inclusion assumption, together with the non triviality assumption, imply that the nearly-periodic torus $\mathcal{D}_{r_j,s_j}(\omega_j)$ also describes a neighbourhood of a resonance, but of multiplicity $n-j$. Note that such a condition will put an important restriction on our choice of the sequence $(\omega_1,\dots,\omega_n)$ as they will have to be sufficiently close to each other to ensure these inclusions. 

Then, the condition on our parameter $m\in\N$, 
\[ mT_jr_j\PM s_j, \quad j\in\{1,\dots,n\},  \]
is also important as it will later determine $m$ in terms of $\varepsilon$ and hence the precise size of the exponentially small term. 

Finally, the other conditions are only technical (and will be easily arranged in the sequel), as they only give smallness conditions on~$\varepsilon$.

\paraga Our normal form is described in the next proposition.

\begin{proposition} \label{lemmeham}
Consider $H=h+f$ as in (\ref{Ham}) and let $j\in\{1,\dots,n\}$. If $(A_i)$ is satisfied for any $i\in\{1,\dots,j\}$, then there exists an analytic symplectic transformation 
\[ \Psi_j: \mathcal{D}_{2r_j/3,2s_j/3}(\omega_j) \rightarrow \mathcal{D}_{r_1,s_1}(\omega_1)\] 
such that 
\begin{equation*}
H \circ \Psi_j=h+g_j+f_j,
\end{equation*}
with $\{g_j,l_i\}=0$ for $i \in \{1, \dots, j\}$ and the estimates
\begin{equation*}
|\partial_\theta g_j|_{2r_j/3,2s_j/3} \MP \varepsilon, \quad |\partial_\theta f_j|_{2r_j/3,2s_j/3} \MP e^{-m}\varepsilon. 
\end{equation*}
Moreover, we have $\Psi_j=\Phi_1 \circ \cdots \circ \Phi_j$ with
\[ \Phi_i: \mathcal{D}_{2r_i/3,2s_i/3}(\omega_i) \rightarrow \mathcal{D}_{r_i,s_i}(\omega_i)\] 
such that $|\Phi_i-\mathrm{Id}|_{2r_i/3,2s_i/3} \PM r_i$, for $i\in\{1,\dots,j\}$.
\end{proposition}

The proof of Proposition~\ref{lemmeham} goes by induction, it is not difficult but quite long, and so it is deferred to Appendix~\ref{estimates}. Here we will try to give a sketch in the case $j=2$, explaining the main ideas without any technicalities. 

\bigskip

The first step is to prove the case $j=1$, that is to find a transformation $\Psi_1$ such that $H \circ \Psi_1=h+g_1+f_1$ with $\{g_1,l_1\}=0$ and $f_1$ exponentially small with $m$. This is very classical. First observe that we can write our original Hamiltonian as $H=h+g^0+f^0$, where $g^0=0$ trivially satisfies $\{g^0,l_1\}=0$ and $f^0=f$ is order $\varepsilon$. Now it is easy to produce a transformation $\varphi^0$ such that $H\circ\varphi^0=h+g^1+f^1$, with $\{g^1,l_1\}=0$, but thanks to our assumption $(A_1)$ the remainder $f^1$ can be made smaller, of order $e^{-1}\varepsilon$: this is an averaging process, $g_1=[f^0]_1$ and the remainder is estimated by Cauchy inequality. Now we only have to iterate this process $m$ times, and writing $\Psi_1=\Phi_1=\varphi^0 \circ \dots \varphi^{m-1}$, $g_1=g^m$ and $f_1=f^m$, we end up with $H\circ\Psi_1=h+g_1+f_1$ with the required properties.

For the second step, we use the first one and consider $H\circ\Psi_1=h+g_1+f_1$ which, by our assumption on the inclusion of domains (this is part of $(A_2)$), is also defined on $\mathcal{D}_{r_2,s_2}(\omega_2)$. We can forget for a moment about $f_1$ which is already exponentially small and consider $g_1$ as the new perturbation. Now as in the first step, we can construct a transformation $\Phi_2$ such that $(h+g_1)\circ\Phi_2=h+g_2+f^2$ with $\{g_2,l_2\}=0$ and $f^2$ is exponentially small: we start with $h+g_1=h+g_1^0+f_1^0$, where $g_1^0=0$, $f_1^0=g_1$ and we find $\varphi^1$ such that $(h+g_1)\circ\varphi^1=h+g_1^1+f_1^1$ where $g_1^1=[f_1^0]_2$. After $m$ iterations we finally have $g_2=g_1^m$ and $f^2=f_1^m$. Assuming we still have $\{g_2,l_1\}=0$, the conclusion follows: let $\Psi_2=\Psi_1 \circ \Phi_2=\Phi_1 \circ \Phi_2$ and $f_2=f^2+f_1\circ\Phi_2$, then $H\circ\Psi_2=h+g_2+f_2$ has the desired properties.  

So it remains to explain why $\{g_2,l_1\}=0$. The key observation is the following: if $g_1$ satisfies $\{g_1,l_1\}=0$, then   
\[ [g_1]_2=\frac{1}{T_2}\int_{0}^{T_2}g_1\circ\Phi_{s}^{l_2}ds \]
and 
\[ \chi=\frac{1}{T_2}\int_{0}^{T_2}(g_1-[g_1]_2)\circ\Phi_{s}^{l_2}sds \]
also satisfy $\{[g_1]_2,l_1\}=0$ and $\{\chi,l_1\}=0$. In Appendix~\ref{estimates}, this will be done by direct computations, but this is in fact a more general phenomenon in normal form theory and it is not restricted to the situation we consider here. Indeed, since $\{l_1,l_2\}=0$, the linear operators $L_{l_1}=\{.,l_1\}$ and $L_{l_2}=\{.,l_2\}$ commutes, so that the kernel of $L_{l_1}$ is invariant by $L_{l_2}$, and as $L_{l_2}$ is semi-simple, it is also invariant under the projection onto the kernel of $L_{l_2}$ which is given by the map $[.]_2$. This explains why $\{[g_1]_2,l_1\}=0$. Now $g_1-[g_1]_2$ is in the kernel of $L_{l_1}$, and its unique pre-image by $L_{l_2}$ is given by $\chi$, hence $\{\chi,l_1\}=0$. 

\begin{remark}
Note that this property was actually used by Bambusi (\cite{Bam99}, Lemma 8.4).
\end{remark}

\paraga Let us now examine the dynamical consequences of our normal form. As usual, it will be used to control the directions, if any, in which the action variables in these new coordinates can actually drift, and we shall come back to our original coordinates at the beginning of section~\ref{ss32}.  

Under the assumptions of Proposition~\ref{lemmeham}, consider the Hamiltonian
\[ H_j=H\circ\Psi_j=h+g_j+f_j \]  
on the domain $\mathcal{D}_{2r_j/3,2s_j/3}(\omega_j)$. Let $\mathcal{M}_j$ be the $\Z$-module
\[ \mathcal{M}_j=\{k\in\Z^n \; | \; k.\omega_i=0, \; i\in\{1,\dots,j\}\}, \]
whose rank is $n-j$, and $\Lambda_j=\mathcal{M}_j \otimes \R$ the vector space spanned by $\mathcal{M}_j$.

The following lemma is completely obvious using the definition of the Poisson bracket.

\begin{lemma} \label{resonant}
The equality $\{g_j,l_i\}=0$, for all $i\in\{1,\dots,j\}$, is equivalent to $\partial_\theta g_j \in \Lambda_j$.
\end{lemma}
 
Now consider a solution $(\theta^j(t),I^j(t))$ of $H_j$ with an initial action $I^j(t_j) \in \mathcal{B}_{2r_j/3}(\omega_j)$ for some $t_j \in \R$, and define the time of escape of this solution as the smallest time $\tilde{t}_j \in ]t_j,+\infty]$ for which $I^j(\tilde{t}_j) \notin \mathcal{B}_{2r_j/3}(\omega_j)$. The only information we shall use from our normal form is contained in the next proposition.    

\begin{proposition} \label{cor}
Let $\Pi_j$ be the projection onto the linear subspace $\Lambda_j$, then with the previous notations, we have
\[ |I^j(t)-I^j(t_j)-\Pi_j(I^j(t)-I^j(t_j))| \MP \varepsilon, \quad t \in [t_j,e^m[ \cap [t_j,\tilde{t}_j[. \] 
In particular, 
\[ |I^n(t)-I^n(t_n)| \MP \varepsilon, \quad t \in [t_n,e^m[. \]
\end{proposition}

\begin{proof}
Let $\Pi_{j}^{\perp}$ be the projection onto the orthogonal complement of $\Lambda_j$, so that $\Pi_j+\Pi_{j}^{\perp}$ is the identity and therefore 
\[ |I^j(t)-I^j(t_j)-\Pi_j(I^j(t)-I^j(t_j))|=|\Pi_{j}^{\perp}(I^j(t)-I^j(t_j))|. \]
Now, as long as $t<\tilde{t}_j$, the equations of motion for $H_j=h+g_j+f_j$ and the mean value theorem give
\[ |I^j(t)-I^j(t_j)| \leq |t-t_j||\partial_\theta (g_j+f_j)|_{2r_j/3,2s_j/3}.  \]
But $\{g_j,l_i\}=0$ for $i\in\{1,\dots,j\}$, so by Lemma~\ref{resonant} we have $\partial_\theta g_j \in \Lambda_j$, hence if we first project the equations onto the orthogonal complement of $\Lambda_j$ we have
\[ |\Pi_{j}^{\perp}(I^j(t)-I^j(t_j))| \leq |t-t_j||\partial_\theta f_j|_{2r_j/3,2s_j/3}.  \]
Now since $|t-t_j| < e^m$ and $|\partial_\theta f_j|_{2r_j/3,2s_j/3} \MP e^{-m}\varepsilon$, the previous estimate gives 
\[ |\Pi_{j}^{\perp}(I^j(t)-I^j(t_j))| \MP \varepsilon,  \]
and therefore
\[ |I^j(t)-I^j(t_j)-\Pi_j(I^j(t)-I^j(t_j))| \MP \varepsilon \]
for $t \in [t_j,e^m[ \cap [t_j,\tilde{t}_j[$. 

Finally, note that $\Pi_{n}$ is identically zero, so that the mean value theorem immediately gives $\tilde{t}_n \geq e^m$ and the estimate
\[ |I^n(t)-I^n(t_n)| \MP \varepsilon, \quad t \in [t_n,e^m[, \]
follows easily. This concludes the proof.
\end{proof}

The interpretation of the above proposition is the following: if $\lambda_j$ is the affine subspace passing through $I^j(t_j)$ with direction space $\Lambda_j$, then as long as $I^j(t)$ remains in the domain $\mathcal{B}_{r_j}(\omega_j)$, it is $\varepsilon$-close to $\lambda_j$ for an exponentially long time with respect to $m$. This means that for that interval of time, there is almost no variation of the action components in the direction transversal to $\lambda_j$, so that any potential drift has to occur along that space.  

\subsection{Geometric part.} \label{ss32}

In this section we will give the proof of Theorem~\ref{mainth2} using the method introduced by Niederman in \cite{Nie04} and \cite{Nie07}. Without loss of generality, we will consider only solutions $(\theta(t),I(t))$ starting at time $t_0=0$ and evolving in positive time $t>0$. We will first show that some specific solutions are exponentially stable, but to define them we shall need some extra notations.

\paraga Consider a sequence of linearly independent periodic vectors $(\omega_1,\dots,\omega_n)$, with periods $(T_1,\dots,T_n)$, and two decreasing sequences of real numbers $(r_1,\dots,r_n)$ and $(s_1,\dots,s_n)$ satisfying conditions $(A_j)$, for $j\in\{1,\dots,n\}$. Recall that from Proposition~\ref{lemmeham} we have a transformation 
\[ \Psi_j: \mathcal{D}_{2r_j/3,2s_j/3}(\omega_j) \rightarrow \mathcal{D}_{r_1,s_1}(\omega_1), \quad j \in \{1,\dots,n\},\]
such that $\Psi_j=\Phi_1 \circ \cdots \circ\Phi_j$, where
\[ \Phi_i: \mathcal{D}_{2r_i/3,2s_i/3}(\omega_i) \rightarrow \mathcal{D}_{r_i,s_i}(\omega_i), \quad i \in \{1,\dots,j\},\]
satisfies the estimate $|\Phi_i-\mathrm{Id}|_{2r_i/3,2s_i/3} \PM r_i$. 

By construction, our transformations preserve reality so that
\[ \Phi_i: \T^n\times\mathcal{B}_{2r_i/3}(\omega_i) \rightarrow \T^n\times\mathcal{B}_{r_i}(\omega_i), \quad i \in \{1,\dots,j\},\]
with  $|\Phi_i-\mathrm{Id}|_{2r_i/3} \PM r_i$. In particular, arranging the implicit constant in the previous estimate ensure that the image of $\mathcal{B}_{2r_i/3}(\omega_i)$ under $\Phi_i$ contains the smaller domain $\mathcal{B}_{r_i/3}(\omega_i)$. From now on, we shall simply write
\[ \mathcal{B}_i=\mathcal{B}_{r_i/3}(\omega_i), \quad i\in\{1,\dots,n\}, \]
and for completeness $\mathcal{B}_0=B$.

Given a solution $(\theta(t),I(t))\in B$ starting at time $t_0=0$, we can define inductively the ``averaged" solution $(\theta^i(t),I^i(t))$ for $i\in\{1,\dots,n\}$ by
\[ \Phi_i(\theta^i(t),I^i(t))=(\theta^{i-1}(t),I^{i-1}(t)) \]
as long as $I^{i-1}(t)\in\mathcal{B}_i$, with $(\theta^0(t),I^0(t))=(\theta(t),I(t))$. Moreover, using our estimate on $\Phi_i$ we have
\begin{equation}\label{truc}
|I^i(t)-I^{i-1}(t)| \PM r_i,  \quad i\in\{1,\dots,n\},
\end{equation}
during that time interval.

\paraga We can finally make our definition.

\begin{definition}
Given $r_0>0$ and $m\in\N$, a solution $(\theta(t),I(t))$ of the Hamiltonian~(\ref{Ham}), starting at time $t_0=0$, is said to be {\it restrained} (by $r_0$, up to time $e^m$) if we can find sequences of: 
\begin{itemize}
\item[(1)] radii $(r_1,\dots,r_n)$, with $0<r_n<\cdots<r_1<r_0$;
\item[(2)] widths $(s_1,\dots,s_n)$, with $0<s_n<\cdots<s_1$; 
\item[(3)] independent periodic vectors $(\omega_1,\dots,\omega_n)$, with periods $(T_1,\dots,T_n)$; 
\item[(4)] times $(t_1,\dots,t_n)$, with $0=t_0 \leq t_1 \leq \cdots \leq t_n \leq t_{n+1}=e^m$, 
\end{itemize}
satisfying, for $j\in\{0,\dots,n-1\}$, conditions $(A_{j+1})$ and the following conditions~$(B_j)$ defined by
\begin{equation} \label{Bj}
\begin{cases}
|I^j(t)-I^j(t_j)| < r_j, \quad \ t \in [t_j,t_{j+1}], \\ 
|\nabla h(I^j(t_{j+1}))-\omega_{j+1}|<r_{j+1}.
\end{cases} \tag{$B_j$}
\end{equation}
\end{definition}

Before explaining this definition, we need to make several remarks. First, for $j\in\{0,\dots,n-2\}$ we will see that the first condition of $(B_{j+1})$ is well defined by the second condition of $(B_j)$. Furthermore, for $j\in\{0,\dots,n-1\}$ the last condition in $(B_j)$ implies in particular that the set $\mathcal{B}_{j+1}(\omega_{j+1})$ is non-empty so we may remove this assumption from $(A_{j+1})$. Finally, we can choose the same sequence of widths $(s_1,\dots,s_n)$ for all solutions, therefore we may already fix $s_i \PE s$ with a suitable constant and this simplifies some conditions (for instance, the condition $mT_jr_j \PM s_j$ appearing in $(A_{j})$ will be replaced by $mT_jr_j \PM 1$).

We have chosen the word ``restrained" because for such a solution the actions $I(t)$ (or some properly normalized actions $I^{j}(t)$) are forced to pass close to a resonance at the time $t=t_j$, the multiplicity of which decreases as $j$ increases, and moreover the variation of these (normalized) actions is controlled on each time interval $[t_j,t_{j+1}]$. Hence after the time $t_n$, the actions are in a domain free of resonances and they are easily confined in view of the last part of Proposition~\ref{cor}. This is reminiscent of the original mechanism of Nekhoroshev, but the fact that we consider each solution individually will greatly simplify this geometric part.

\paraga Let us see how the actions of a restrained solution are easily confined for an exponentially long time with respect to $m$. We shall write
\[ \rho_j=r_1+\cdots+r_j, \]
for $j\in\{1,\dots,n\}$.

\begin{proposition} \label{etape1}

Consider a restrained solution $(\theta(t),I(t))$, with an initial action $I(0)\in B_{R/2}$. If 
\begin{itemize}
\item[$(i)$] $\varepsilon \PM r_n$;
\item[$(ii)$] $r_0 \PM R$,
\end{itemize}
then the estimates
\[ |I(t)-I(0)| < (n+1)^2 r_0, \quad 0\leq t< e^m, \]
hold true.
\end{proposition}

\begin{proof}
First observe that for each $j \in \{1,\dots,n-1\}$, for $t \in [t_j,t_{j+1}]$ we have  
\[ |I(t)-I(t_j)| \leq |I(t)-I^j(t)|+|I^j(t)-I^j(t_j)|+|I^j(t_{j})-I(t_j)|,  \]
so the first part of $(B_j)$ and (\ref{truc}) yields
\begin{equation} \label{p1}
|I(t)-I(t_j)|< 2\rho_j +r_j
\end{equation}  
while for $t\in [0,t_1]$, the first part of $(B_0)$ reads 
\begin{equation}
|I(t)-I(0)| < r_0.  \label{p2}
\end{equation}
Now let $t \in [0,e^m]$, then $t \in [t_j,t_{j+1}]$ for some $j \in \{0,\dots,n\}$ (recall that $t_{n+1}=e^m$), and we will distinguish three cases. 

First assume that $t \in [0,t_1]$, in this case the conclusion follows by~(\ref{p2}) since $(n+1)^2 \geq 1$. Now assume $t \in [t_j,t_{j+1}]$ for some $j \in \{1,\dots,n-1\}$, then we can write
\[ |I(t)-I(0)| \leq |I(t)-I(t_j)| +\sum_{i=0}^{j-1}|I(t_{i+1})-I(t_i)|, \]
and by (\ref{p1}) and (\ref{p2})
\[ |I(t)-I(0)| < \sum_{i=1}^{j}(2\rho_i+r_i)+r_0 <(n+1)^2 r_0, \]
since $r_i<r_0$ for $i\in \{1,\dots,j\}$. Finally, assume that $t \in [t_n,t_{n+1}]$, then we can apply the second inequality of Proposition~\ref{cor} and $(i)$ to estimate
\[ |I^n(t)-I^n(t_n)| \MP \varepsilon< r_n, \]
and so 
\[ |I(t)-I(t_n)|< 2\rho_n +r_n \]
which gives
\[ |I(t)-I(0)| < \sum_{i=1}^{n}(2\rho_i+r_i)+r_0 <(n+1)^2 r_0. \]
To conclude, just note $I(0) \in B_{R/2}$ and $(ii)$ ensures that $I(t)$ remains in $B_R$ for $t<e^m$.
\end{proof}

\paraga Restrained solutions are exponentially stable, and now we will show that this is in fact true for all solutions. However, to use our steepness arguments this will be done quite indirectly, and so it is useful to introduce the following definition.

\begin{definition}
Given $r_0>0$ and $m\in\N$, a solution $(\theta(t),I(t))$ of the Hamiltonian~(\ref{Ham}), starting at time $t_0=0$, is said to be drifting (to $r_0$, before time $e^m$) if there exists a time $t_*$ satisfying  
\[ |I(t_*)-I(0)|=(n+1)^2 r_0, \quad 0<t_*< e^m. \]
\end{definition} 

Of course, this definition makes sense only if $(n+1)^2 r_0<R/2$. In view of Proposition~\ref{etape1}, drifting solutions cannot be restrained. However, we will prove below that if such a drifting solution exists, it has to be restrained under some assumptions on $r_0$, $m$ and $\varepsilon$, which will eventually prove that all solutions are in fact exponentially stable. 

More precisely, assuming the existence of a drifting solution, we will construct a sequence of radii $(r_1,\dots,r_n)$, an increasing sequence of times $(t_1,\dots,t_n)$ and a sequence of linearly independent vectors $(\omega_1,\dots,\omega_n)$, with periods $(T_1,\dots,T_n)$ satisfying, for $j\in\{0,\dots,n-1\}$, assumptions $(A_{j+1})$ and $(B_j)$. All sequences will be built inductively, and we first describe the tools that we shall need. 

\paraga For $j\in\{1,\dots,n\}$, recall that $\Lambda_j$  is the vector space spanned by 
\[ \mathcal{M}_j= \{k\in\Z^n \; | \; k.\omega_i=0, \; i\in\{1,\dots,j\}\}, \] 
and that $\Pi_j$ (resp. $\Pi_{j}^{\perp}$) is the projection onto $\Lambda_j$ (resp. $\Lambda_{j}^{\perp}$). Let us define the integer 
\[ L_j=\sup_{i\in\{1,\dots,j\}}\{|T_i\omega_i|\}\in\N^*, \quad j\in\{1,\dots,n\}. \]
For completeness, we set $\Lambda_0=\R^n$, $L_0=1$ and in this case $\Pi_0$ is nothing but the identity. To construct the sequence of times, we will rely on the fact that our integrable part $h$ belongs to $SDM_{\gamma}^{\tau}(B)$, so that it satisfies the following steepness property (see Appendix~\ref{SDM}).

\begin{lemma} \label{l1}
For $j\in\{0,\dots,n-1\}$, let $\lambda_j$ be any affine subspace with direction $\Lambda_j$, and take $r<1$. Then for any continuous curve $\Gamma : [0,1] \rightarrow \lambda_j \cap B$ with length 
\[ \vert\Gamma(0)-\Gamma(1)\vert =r \PM \gamma L_j^{-\tau},\] 
there exists a time $t_*\in [0,1]$ such that 
\begin{equation*}
\begin{cases}
\vert\Gamma (t)-\Gamma (0)\vert < r, \quad t\in [0,t_*], \\
\left\vert\Pi_j(\nabla h(\Gamma (t_*)))\right\vert \PS r^2. 
\end{cases}
\end{equation*}
\end{lemma} 

\begin{proof}
For any $j\in\{1,\dots,n-1\}$, the orthogonal complement of $\Lambda_j$ is spanned by $\omega_1, \dots, \omega_j$, hence by the integer vectors $T_1\omega_1, \dots, T_j\omega_j$, so that $\Lambda_j$ belongs to $G^{L_j}(n,n-j)$ with the integer $L_j$ defined above. Therefore one can apply the Proposition~\ref{steep} in Appendix~\ref{SDM} to get the required properties (note that here we are using the supremum norm instead of the Euclidean norm, so the implicit constants are different).

For $j=0$, $\Gamma : [0,1] \rightarrow B=B \cap \R^n$, but since the orthogonal complement of $\R^n$ is trivial one can take $L_0=1$. 
\end{proof}

\paraga To construct the sequence of periodic vectors, we shall use the following lemma, which is a straightforward application of Dirichlet's theorem on simultaneous Diophantine approximation (see \cite{Cas57}).

\begin{lemma} \label{l2}
Given any vector $v \in \R^n$ and any real number $Q>0$, there exists a $T$-periodic vector $\omega$ satisfying
\[ |v-\omega| \leq T^{-1}Q^{-\frac{1}{n-1}}, \quad |v|^{-1} \leq T \leq Q|v|^{-1}. \]
\end{lemma} 

\begin{proof}
Fix any real number $Q>0$. We can write the vector $v$, up to re-ordering its components, as $v=|v|(\pm 1,x)$ with $x \in \R^{n-1}$, and it will be enough to approximate $x$ by a periodic vector. By a theorem of Dirichlet, we can find an integer $q$, with $1\leq q<Q$, such that
\[ |qx-p| \leq Q^{-\frac{1}{n-1}}, \]
for some $p\in\Z^{n-1}$. The vector $q^{-1}p$ is trivially $q$-periodic, hence the vector $\omega=|v|(\pm 1,q^{-1}p)$ is $T$-periodic, with $T=|v|^{-1}q$, therefore
\[ |v|^{-1} \leq T \leq Q|v|^{-1}, \]  
and we have the estimate
\[ |v-\omega|\leq T^{-1}|qx-p| \leq T^{-1}Q^{-\frac{1}{n-1}}. \]   
\end{proof}

\paraga Now we can finally prove that drifting solutions are in fact restrained under some assumptions. This will be done inductively, and for technical reasons we separate the first step (Proposition~\ref{etape2}) from the general inductive step (Proposition~\ref{etape3}).

\begin{proposition}\label{etape2}
Let $(\theta(t),I(t))$ be a drifting solution. If $r_0 \PM \gamma$, then there exist a time $t_1$, a $T_1$-periodic vector $\omega_1$ and $r_1 \EP T_1^{-1}\varepsilon^{a_1}$ for some constant $a_1$, satisfying $(B_0)$. Moreover, we have the estimates
\begin{equation}\label{estimper}
1 \MP T_1 \MP \varepsilon^{-a_{1}(n-1)}r_{0}^{-2}, \quad 1 \leq L_1 \MP \varepsilon^{-a_{1}(n-1)}r_{0}^{-2}.
\end{equation} 
\end{proposition}

\begin{proof}
We need to construct $t_1$, $\omega_1$ and $r_1$ satisfying
\begin{itemize}
\item[$(a)$] $|I(t)-I(0)| < r_0,\quad t \in [0,t_1]$; 
\item[$(b)$] $|\nabla h(I(t_1))-\omega_1| < r_1$,
\end{itemize}
and the estimate~(\ref{estimper}). Consider the curve 
\[ \Gamma_1: t \in [0,t_*] \longmapsto I(t) \in B \subseteq \R^n. \]
Since we have a drifting solution, we can select $t_0^* \in [0,t_*]$ such that 
\[ |\Gamma_1(t_0^*)-\Gamma_1(0)|=r_0.\]  
Now using the fact that $h \in SDM_{\gamma}^{\tau}(B)$ and $r_0 \PM \gamma$ (recall that $L_0=1$), we can apply Lemma~\ref{l1} (the case $j=0$) to the curve $\Gamma_1$ restricted to $[0,t_0^*]$ to find a time $t_1 \in [0,t_0^*]$ for which
\begin{equation} \label{steep0}
\begin{cases}
|I(t)-I(0)| < r_0, \quad t\in [0,t_1], \\
\vert\nabla h(I(t_1)) \vert \PS r_0^2.
\end{cases}
\end{equation} 
The first inequality of~(\ref{steep0}) gives $(a)$.

Now choose $Q_1=\varepsilon^{-a_1(n-1)}$, for some constant $a_1$ yet to be chosen, and apply Lemma~\ref{l2} to approximate $\nabla h(I(t_1))$ by a $T_1$-periodic vector $\omega_1$, that is
\begin{equation}\label{dir1}
|\nabla h(I(t_1))-\omega_1| \leq T_{1}^{-1}Q_1^{-\frac{1}{n-1}}=T_{1}^{-1}\varepsilon^{a_1}. 
\end{equation}
Moreover, since
\[ r_0^2 \MP |\nabla h(I(t_1))| \MP 1, \]
the period $T_1$ satisfies the following estimate
\begin{equation}\label{der1}
1 \MP T_1 \MP \varepsilon^{-a_1(n-1)} r_{0}^{-2}.
\end{equation}  
Now choose $r_1 \EP T_{1}^{-1}\varepsilon^{a_1}$ so that~(\ref{dir1}) gives $(b)$. Finally, as $L_1=|T_1\omega_1|$ and
\[ |\omega_1|\leq |\nabla h(I(t_1))|+|\nabla h(I(t_1))-\omega_1|\MP 1  \]
we obtain
\begin{equation}\label{der2}
1 \leq L_1 \MP \varepsilon^{-a_1(n-1)} r_{0}^{-2}
\end{equation} 
where the lower bound follows from the fact that $T_1\omega_1$ is a non-zero integer vector. The estimates~(\ref{der1}) and~(\ref{der2}) give~(\ref{estimper}).
\end{proof}

\begin{proposition} \label{etape3}
Let $(\theta(t),I(t))$ be a drifting solution, $j \in \{1,\dots,n-1\}$ and assume that there exist sequences $(t_1,\dots,t_j)$, $(\omega_1,\dots,\omega_j)$ linearly independent and $(r_1, \dots, r_j)$, satisfying assumptions $(A_{i})$ and $(B_{i-1})$, for $i \in \{1,\dots,j\}$. Assume also that
\begin{itemize}
\item[$(i)$] $r_j \PM \min\{r,s\}$;
\item[$(ii)$] $mT_j\varepsilon\PM r_1r_j$;
\item[$(iii)$] $mT_jr_j \PM 1$ ;
\item[$(iv)$] $\left(T_jr_jL_j^{-1}\right)^\tau \PM \gamma L_j^{-\tau}$;
\item[$(v)$] $\left(T_jr_jL_j^{-1}\right)^\tau \PM r_j$;
\item[$(vi)$] $\varepsilon \PM \left(T_jr_jL_j^{-1}\right)^{2\tau}$;
\item[$(vii)$] $r_1 \PM r_0^2$.
\end{itemize}
Then there exist a time $t_{j+1}$, a $T_{j+1}$-periodic vector $\omega_{j+1}$ and $r_{j+1} \EP T_{j+1}^{-1}\varepsilon^{a_{j+1}}$ for some constant $a_{j+1}$, satisfying $(A_{j+1})$ and $(B_j)$. Moreover, we have the estimates
\begin{equation}\label{estimper2}
1 \MP T_{j+1} \MP \varepsilon^{-a_{j+1}(n-1)}r_{0}^{-2}, \quad 1 \leq L_{j+1} \MP \max_{i\in\{1,\dots,j+1\}}\{\varepsilon^{-a_i(n-1)}\}r_0^{-2},
\end{equation}  
and if
\begin{itemize}
\item[$(viii)$] $r_{j+1} \PM \left(T_jr_{j}L_j^{-1}\right)^{2\tau}$,
\end{itemize}
then $\omega_{j+1}$ is linearly independent of $(\omega_1,\dots,\omega_j)$.
\end{proposition}

\begin{proof}
First note that for $j=1$, we do not require that $t_1$, $\omega_1$ and $r_1$ satisfy $(A_1)$ since this is implied by the conditions $(i),(ii)$ and $(iii)$, and for $j>1$, the same conditions reduce assumption $(A_{j+1})$ to the inclusion of real domains $\mathcal{B}_{r_{j+1}}(\omega_{j+1}) \subseteq \mathcal{B}_{2r_j/3}(\omega_j)$ (recall that by condition $(B_{j-1})$ these domains are non-empty, and that we have already fixed $s_j \PE s$).

Therefore, we need to construct $t_{j+1}$, $\omega_{j+1}$ and $r_{j+1}$ satisfying 
\begin{itemize}
\item[$(a)$] $|I^j(t)-I^j(t_j)| < r_j,\quad t \in [t_j,t_{j+1}]$;
\item[$(b)$] $|\nabla h(I^j(t_{j+1}))-\omega_{j+1}| < r_{j+1}$;
\item[$(c)$] $\omega_{j+1}$ is independent of $(\omega_1,\dots,\omega_j)$;
\item[$(d)$] $\mathcal{B}_{r_{j+1}}(\omega_{j+1}) \subseteq \mathcal{B}_{2r_j/3}(\omega_j)$,
\end{itemize}
and the estimates~(\ref{estimper2}).

Let $\tilde{t}_j$ be the maximal time of existence within $\mathcal{B}_j$ of the solution $I^j(t)$ starting at $I^j(t_j)$. Since $(A_j)$ is satisfied, we can apply Proposition~\ref{cor} and for $t \in [t_j,\tilde{t}_j] \cap [t_j,e^m]$, we have
\begin{equation} 
|I^j(t)-I^j(t_j)-\Pi_{j}(I^j(t)-I^j(t_j))|\MP \varepsilon. \label{main}
\end{equation} 
Now consider the curve
\[ \Gamma_{j+1}: t \in [t_j,\tilde{t}_j] \cap [t_j,e^m] \longmapsto I^j(t_j)+\Pi_j(I^j(t)-I^j(t_j)) \in \lambda_j \cap B,\]
where $\lambda_j$ is the affine subspace $I^j(t_j)+\Lambda_j$.

\bigskip

\textit{Claim: there exists a time $t_j^* \in [t_j,\tilde{t}_j] \cap [t_j,e^m]$ such that
\[ |\Gamma_{j+1}(t_j^*)-\Gamma_{j+1}(t_j)|=|\Pi_j(I^j(t_j^*)-I^j(t_j))|= \left(T_jr_jL_j^{-1}\right)^\tau. \]}

Let us prove the claim. We have to distinguish two cases.

\bigskip

\textit{First case: $\tilde{t}_j \leq e^m$.}
We have 
\[ |\nabla h(I^j(t_j))-\omega_j| \leq |\nabla h(I^j(t_j))-\nabla h(I^{j-1}(t_j))| + |\nabla h(I^{j-1}(t_j))-\omega_j|,  \]
and therefore
\[ |\nabla h(I^j(t_j))-\omega_j| \MP r_j,  \]
while by definition, 
\[ |\nabla h(I^j(\tilde{t}_j))-\omega_j| \EP r_j \]
with a sufficiently {\it larger} implicit constant (see Remark~\ref{defdomain}). Hence
\[ |\nabla h(I^j(\tilde{t}_j))-\nabla h(I^j(t_j))| \SP r_j, \] 
and this implies
\begin{equation} \label{ineq}
|I^j(\tilde{t}_j)-I^j(t_j)| \SP r_j. 
\end{equation}
But conditions $(v)$ and $(vi)$ give in particular
\[ \varepsilon \PM r_j, \]
so that~(\ref{main}) and~(\ref{ineq}) yields
\[ |\Pi_j(I^j(\tilde{t}_j)-I^j(t_j))| \SP r_j. \]
Now using $(v)$ again, this gives 
\[ |\Pi_j(I^j(\tilde{t}_j)-I^j(t_j))| \SP \left(T_jr_jL_j^{-1}\right)^\tau,  \]
and so we can certainly find a time $t_j^* \in [t_j,\tilde{t}_j]$ such that
\[|\Pi_j(I^j(t_j^*)-I^j(t_j))|= \left(T_jr_jL_j^{-1}\right)^\tau.\]

\bigskip

\textit{Second case: $\tilde{t}_j > e^m$.}
We will first prove that $t_* \in [t_j,e^m]$. Indeed, otherwise $t_*$ belongs to $[t_k,t_{k+1}]$ for some $k\in\{0,\dots,j-1\}$ and we can write
\begin{equation} \label{contra}
|I(t_*)-I(0)| \leq |I(t_*)-I(t_k)| + \sum_{i=0}^{k-1}|I(t_{i+1})-I(t_{i})|.   
\end{equation}
Each term of the right-hand side of~(\ref{contra}) is easily estimated: using $(B_i)$ for $i \in \{0,\dots,k-1\}$ we have
\[ |I^i(t_{i+1})-I^i(t_i)| < r_i, \quad |I^k(t_*)-I^k(t_k)| < r_k,  \]
which implies, by the triangle inequality and the estimate~(\ref{truc})
\[ |I(t_{i+1})-I(t_i)| < 2\rho_i +r_i, \quad |I(t_*)-I(t_k)| < 2\rho_k +r_k. \]
Moreover, 
\[ |I(t_1)-I(0)| < r_0, \]
hence we find
\[ |I(t_*)-I(0)| < \sum_{i=1}^{k} (2\rho_i +r_i)+r_0 < (n+1)^2 r_0,\]
which of course contradicts the definition of our drifting time $t_*$. 

Now to prove the claim, we argue by contradiction and suppose that 
\[ |\Pi_j(I^j(t)-I^j(t_j))|<\left(T_jr_jL_j^{-1}\right)^\tau, \quad t \in [t_j,e^m].\]
Since $t_*\in [t_j,e^m]$, we can use the previous inequality together with the estimate~(\ref{main}) and both conditions $(v)$ and $(vi)$ to first obtain
\[ |I^j(t_*)-I^j(t_j)| < r_j, \] 
and then with the triangle inequality
\[ |I(t_*)-I(t_j)| < 2\rho_j +r_j. \] 
Now, as the argument above, writing
\begin{equation*} 
|I(t_*)-I(0)| \leq |I(t_*)-I(t_j)| + \sum_{i=0}^{j-1}|I(t_{i+1})-I(t_{i})|   
\end{equation*}
we find the same contradiction on the time $t_*$, which completes the proof of the claim.

\bigskip

Now consider the restriction of the curve $\Gamma_{j+1}$ on the interval $[t_j,t_j^*]$. Using our claim together with conditions $(iv)$ and $(v)$, we can apply Lemma~\ref{l1} to find a time $t_{j+1} \in [t_j,t_j^*]$ such that
\begin{equation} \label{steepj}
\begin{cases}
|\Pi_j(I^j(t)-I^j(t_j))| < \left(T_jr_jL_j^{-1}\right)^\tau, \quad t\in [t_j,t_{j+1}], \\
\left\vert\Pi_j(\nabla h(\Gamma_{j+1}(t_{j+1})))\right\vert \PS \left(T_jr_jL_j^{-1}\right)^{2\tau}.
\end{cases}
\end{equation}
The first inequality of~(\ref{steepj}), together with~(\ref{main}) and conditions $(v)$ and $(vi)$ give
\[ |I^j(t)-I^j(t_j)| < r_j \]
for $t\in [t_j,t_{j+1}]$, hence $(a)$ is verified. Now as in the first step, choose $Q_{j+1}=\varepsilon^{-a_{j+1}(n-1)}$ for some constant $a_{j+1}$ to be chosen later, and apply Lemma~\ref{l2} to approximate $\nabla h(I^j(t_{j+1}))$ by a $T_{j+1}$-periodic vector $\omega_{j+1}$, that is
\begin{equation}\label{dir2}
|\nabla h(I^j(t_{j+1}))-\omega_{j+1}| \leq T_{j+1}^{-1} Q_{j+1}^{-\frac{1}{n-1}}= T_{j+1}^{-1}\varepsilon^{a_{j+1}}. 
\end{equation}
Let $r_{j+1} \EP T_{j+1}^{-1}\varepsilon^{a_{j+1}}$ so that $(b)$ is verified by~(\ref{dir2}). To estimate the period $T_{j+1}$ and the number $L_j$, we need a lower bound for $|\nabla h(I^j(t_{j+1}))|$ and we will use the fact that we have such a lower bound for $|\nabla h(I(t_1)|$ (see the second inequality of~(\ref{steep0})). First note that one has easily
\[ |I^j(t_{j+1})-I(t_1)| \MP r_1, \]
since $r_i<r_1$ for $i\in\{1,\dots,j\}$, and therefore
\[ |\nabla h(I^j(t_{j+1}))-\nabla h(I(t_1))| \MP r_1, \]
so choosing properly the constant in the condition $(vii)$ we can ensure that
\[ |\nabla h(I^j(t_{j+1}))-\nabla h(I(t_1))| \PM r_0^2\]
and hence
\begin{equation}\label{estimategradient}
|\nabla h(I^j(t_{j+1}))| \geq |\nabla h(I(t_1)| - |\nabla h(I^j(t_{j+1}))-\nabla h(I(t_1))| \PS r_0^2. 
\end{equation}
By Lemma~\ref{l2}, this gives the estimate 
\begin{equation} \label{estimateperiod}
1 \MP T_{j+1} \MP\varepsilon^{-a_{j+1}(n-1)}r_{0}^{-2}. 
\end{equation}
Now as $|\omega_j|\MP 1$ this easily implies that
\begin{equation}\label{estima}
1 \leq L_{j+1} \MP \max_{i\in\{1,\dots,j+1\}}\{\varepsilon^{-a_i(n-1)}\}r_0^{-2}.
\end{equation}
The estimates~(\ref{estimateperiod}) and~(\ref{estima}) give~(\ref{estimper2}).

Next having built $r_{j+1}$, we need to check that $\omega_{j+1}$ is independent of $(\omega_1,\dots,\omega_j)$. First, by using the mean value theorem, the estimate~(\ref{main}) and our condition $(vi)$, we have
\begin{equation*}
|\nabla h(I^j(t_{j+1}))-\nabla h(\Gamma_{j+1}(t_{j+1}))| \PM \left(T_jr_jL_j^{-1}\right)^{2\tau},
\end{equation*}
and together with the second estimate of~(\ref{steepj}), this gives
\begin{equation} \label{indep}
|\Pi_j(\nabla h(I^j(t_{j+1})))| \PS \left(T_jr_jL_j^{-1}\right)^{2\tau}.
\end{equation}
Furthermore, using~(\ref{dir2}) 
\begin{equation*}
|\Pi_j(\nabla h(I^j(t_{j+1}))-\omega_{j+1})| \leq |\nabla h(I^j(t_{j+1}))-
\omega_{j+1}| \PM r_{j+1} 
\end{equation*}
hence with $(viii)$, we get
\begin{equation}\label{indep2}
|\Pi_j(\nabla h(I^j(t_{j+1}))-\omega_{j+1})| \PM \left(T_jr_jL_j^{-1}\right)^{2\tau}.
\end{equation}
Now by the estimates~(\ref{indep}) and~(\ref{indep2})
\[ |\Pi_j(\omega_{j+1})| \geq |\Pi_j(\nabla h(I^j(t_{j+1})))| - |\Pi_j(\nabla h(I^j(t_{j+1}))-\omega_{j+1})| \PS \left(T_jr_jL_j^{-1}\right)^{2\tau} \]
and so $\Pi_j(\omega_{j+1})$ is non zero, which means that $\omega_{j+1}$ is not a linear combination of $\{\omega_1,\dots,\omega_j\}$. This proves $(c)$.

Finally we can write
\begin{eqnarray*}
|\omega_{j+1}-\omega_j| & \leq & |\omega_{j+1}-\nabla h(I^j(t_{j+1}))| + |\nabla h(I^j(t_{j+1}))-\nabla h(I^j(t_j))| \\
& + & |\nabla h(I^j(t_j))-\nabla h(I^{j-1}(t_j))| + |\nabla h(I^{j-1}(t_j))-\omega_j|,
\end{eqnarray*} 
and hence
\[ |\omega_{j+1}-\omega_j| \MP (r_j+r_{j+1}) \MP r_j. \]
So given any $I \in \mathcal{B}_{r_{j+1}}(\omega_{j+1})$, we have
\[ |\nabla h(I)-\omega_j| \leq |\nabla h(I)-\omega_{j+1}|+|\omega_{j+1}-\omega_j| \MP r_j, \]
so that $I \in \mathcal{B}_{2r_j/3}(\omega_j)$, which gives $(d)$. This ends the proof.
\end{proof}

\paraga Now we can eventually complete the proof of the main Theorem~\ref{mainth2}.

\begin{proof}[Proof of Theorem~\ref{mainth2}]
As a consequence of Propositions~\ref{etape1}, \ref{etape2} and~\ref{etape3}, we know that
\[ |I(t)-I(0)| < (n+1)^2 r_0, \quad 0\leq t < e^m \]
provided that the parameters $r_0$, $m$ and $\varepsilon$ satisfy the following eleven conditions:
\begin{itemize}
\item[$(i)$] $r_{j+1} \PM \left(T_jr_jL_j^{-1}\right)^{2\tau}$, $j \in \{1,\dots,n-1\}$;
\item[$(ii)$] $\varepsilon \PM \left(T_jr_jL_j^{-1}\right)^{2\tau}$, $j \in \{1,\dots,n-1\}$;
\item[$(iii)$] $\left(T_jr_jL_j^{-1}\right)^\tau \PM r_j$, for $j \in \{1,\dots,n-1\}$;
\item[$(iv)$] $mT_jr_j \PM 1$, for $j \in \{1,\dots,n\}$;
\item[$(v)$] $r_1 \PM r_0^2$;
\item[$(vi)$] $mT_j\varepsilon\PM r_1r_j$, for $j \in \{1,\dots,n\}$;
\item[$(vii)$] $\varepsilon\PM r_n$;
\item[$(viii)$] $\left(T_jr_jL_j^{-1}\right)^\tau\PM\gamma L_j^{-\tau}$, for $j \in \{1,\dots,n-1\}$;
\item[$(ix)$] $r_0\PM \gamma$;
\item[$(x)$] $r_0 \PM R$;
\item[$(xi)$] $r_j \PM \min\{r,s\}$, for $j \in \{1,\dots,n\}$,
\end{itemize}
where $r_j \EP T_{j}^{-1}\varepsilon^{a_j}$, with $a_j$ to be defined for $j\in\{1,\dots,n\}$, and 
\begin{equation}\label{estimperF}
1 \MP T_j \MP \varepsilon^{-a_{j}(n-1)}r_{0}^{-2}, \quad 1 \leq L_j \MP \max_{i\in\{1,\dots,j\}}\{\varepsilon^{-a_i(n-1)}\}r_0^{-2}. 
\end{equation}
So let us choose $m \PE \varepsilon^{-a}$ and $r_0=\varepsilon^b$, for two constants $a$ and $b$ also to be determined.

Using the estimates~(\ref{estimperF}) on the periods $T_j$, $j\in\{1,\dots,n\}$ and the numbers $L_j$, $j\in\{1,\dots,n-1\}$, as well as the form of $r_0$, $r_j$ for $j\in\{1,\dots,n\}$ and $m$, one can see that conditions $(i)$ to $(xi)$ are implied by the following conditions:
\begin{itemize}
\item[$(i')$] $a_{j+1}-2n\tau\left(\max_{i\in\{1,\dots,j\}}\{a_i\}\right) -4\tau b >0$, $j \in \{1,\dots,n-1\}$;
\item[$(ii')$] $1-2n\tau a_j-4\tau b>0$, $j \in \{1,\dots,n-1\}$;
\item[$(iii')$] $(\tau-1)a_j-2b>0$, for $j \in \{1,\dots,n-1\}$;
\item[$(iv')$] $a_j>a$, for $j \in \{1,\dots,n\}$;
\item[$(v')$] $a_1-2b>0$;
\item[$(vi')$] $1-a-(2n-1)a_j-na_1-6b>0$, for $j \in \{1,\dots,n\}$;
\item[$(vii')$] $1-na_n-2b>0$;
\item[$(viii')$] $\varepsilon < \gamma^{(\tau a_j)^{-1}}$, for $j \in \{1,\dots,n-1\}$;
\item[$(ix')$] $\varepsilon < \gamma^{b^{-1}}$;
\item[$(x')$] $\varepsilon < R^{b^{-1}}$;
\item[$(xi')$] $\varepsilon < (\min\{r,s\})^{(na_j+2b)^{-1}}$, for $j \in \{1,\dots,n\}$.
\end{itemize}
So we need to choose constants $a_j$, $j\in\{1,\dots,n\}$, $a$ and $b$ such that the previous conditions are satisfied. First note that by $(i')$, the sequence $a_j$, for $j\in\{1,\dots,n\}$, has to be increasing, hence 
\[ \max_{i\in\{1,\dots,j \}} \{a_i\}=a_j, \quad j\in\{1,\dots,n\}.\] 
Then using $(v')$, we observe that $(i')$ is satisfied if $a_{j+1}=2\tau (n+1)a_j$ for $j\in\{1,\dots,n-1\}$, that is
\[ a_j=(2\tau (n+1))^{j-1}a_1. \]
Now for $(ii')$ to be satisfied, one can choose
\[ a_1=(2\tau (n+1))^{-n}, \]
so $a_j$, for $j\in\{2,\dots,n\}$, is determined by
\[ a_j=(2\tau (n+1))^{-n-1+j}. \]
Then, since $\tau \geq 2$, we may choose
\[ b=3^{-1}a_1 =3^{-1}(2\tau (n+1))^{-n} \] 
and $(iii')$ easily holds. Finally, we may also choose
\[ a=b=3^{-1}(2\tau (n+1))^{-n} \] 
so that $(iv')$ is satisfied. With those values, it is easy to check that $(v')$, $(vi')$ and $(vii')$ holds, recalling that $\tau\geq 2$ and $n\geq 2$. To conclude, just note that $(viii')$, $(ix')$, $(x')$ and $(xi')$ are satisfied if $\varepsilon \leq \varepsilon_0$ with a sufficiently small $\varepsilon_0$ depending on $n,R,r,s,M,\gamma$ and $\tau$. This ends the proof. 
\end{proof}

\appendix

\section{Proof of the normal form} \label{estimates}

In this first appendix we will give the proof of the normal form~\ref{lemmeham}. We will closely follow the method of~\cite{Pos99a} and deduce our result from an equivalent version in terms of vector fields (Proposition~\ref{lemmevecti} below).

\subsection{Preliminary estimates}

Before giving the proof, we will need some general estimates based on the classical Cauchy inequality.

\paraga First consider the case of a function $f$ analytic on some domain $\mathcal{D}_{r,s}$, and recall that
\[ |\partial_\theta f|_{r,s}=\max_{1 \leq i \leq n}|\partial_{\theta_i} f|_{r,s}, \; |\partial_I f|_{r,s}=\max_{1 \leq i \leq n}|\partial_{I_i} f|_{r,s}. \] 
We take $r',s'$ such that $0<r'<r$ and $0<s'<s$. The first estimate is classical, but we repeat the proof for convenience.

\begin{lemma} \label{lemme1}
Under the previous assumptions, we have
\[ |\partial_{I}f|_{r-r',s} < \frac{1}{r'}|f|_{r,s}, \quad |\partial_{\theta}f|_{r,s-s'} < \frac{1}{s'}|f|_{r,s}. \]
\end{lemma}

\begin{proof}
For $x=(\theta,I) \in \mathcal{D}_{r-r',s}$ and any unit vector $v \in \C^n$, consider the function
\[ F_{x,v}: t \in \C \longmapsto f(\theta,I+tv) \in \C. \] 
This function is well-defined and holomorphic on the disc $|t| < r'$, so the classical Cauchy estimate gives
\[ |F_{x,v}'(0)| < \frac{1}{r'}|f|_{r,s}, \] 
from which the inequality for $\partial_I f$ follows easily by optimizing with respect to $x$ and $v$. The estimate for $\partial_\theta f$ is completely similar.
\end{proof}

\paraga Now let $j\in\{1,\dots,n\}$, and let $f$ and $g$ be analytic functions defined on the domain 
\[ \mathcal{D}_{r_j,s_j}(\omega_j)=\{(\theta,I) \in \mathcal{D}_{r_j,s_j} \; | \; |\nabla h(I)-\omega_j| \MP r_j\}, \]
where $\omega_j$ is a periodic vector. We can define a vector field norm on $\mathcal{D}_{r_j,s_j}(\omega_j)$ by  
\[ |X_f|_{r_j,s_j}=\max\left(|\partial_I f|_{r_j,s_j},|\partial_\theta f|_{r_j,s_j}\right). \]
However, it will more convenient to use the following ``weighted" norm
\[ ||X_f||_{r_j,s_j}=\max\left(|\partial_I f|_{r_j,s_j},s_1r_1^{-1}|\partial_\theta f|_{r_j,s_j}\right), \]
since the components $|\partial_I f|_{r_j,s_j}$ and $|\partial_\theta f|_{r_j,s_j}$ may have very different sizes when estimated from the size of $f$ by a Cauchy estimate (this idea is used in \cite{DG96}).

\begin{remark}
Note that under assumption $(A)$, $s_1r_1^{-1}>1$, so we have the inequality $|X_f|_{r_j,s_j} \leq ||X_f||_{r_j,s_j}$ and the equality holds if $f$ is integrable. Moreover, note that each norm $||\,.\,||_{r_j,s_j}$ is normalized with $s_1r_1^{-1}$ (and not with $s_jr_j^{-1}$): by our inclusions of domains, this implies in particular that $||\,.\,||_{r_{j+1},s_{j+1}} \leq ||\,.\,||_{2r_j/3,2s_j/3}$. 
\end{remark}

It is well-known how to use the Cauchy inequality to estimate the size of the Poisson bracket $\{f,g\}$ in terms of $f$ and $g$. Similarly, our second estimate is concerned with the size of the vector field $[X_f,X_g]$ in terms of $X_f$ and $X_g$. We take $r',s'$ such that $0<r'<r_j$ and $0<s'<s_j$.

\begin{lemma} \label{lemme3}
Under the previous assumptions, we have
\begin{equation*}
||[X_f,X_g]||_{r_j-r',s_j-s'} < \frac{1}{r'}||X_f||_{r_j,s_j}||X_g||_{r_j,s_j}, 
\end{equation*}
and moreover, if $g$ is integrable, then
\begin{equation*}
||[X_f,X_g]||_{r_j-r',s_j-s'} < \frac{1}{s'}||X_f||_{r_j,s_j}||X_g||_{r_j}. 
\end{equation*}
\end{lemma}

\begin{proof}
First recall that 
\[ [X_f,X_g]=\left.\dfrac{d}{dt} (\Phi_t^g)^*X_f \right\vert_{t=0}. \] 
Now fix $x \in \mathcal{D}_{r_j-r',s_j-s'}$, and let us define the vector-valued function
\[ F_x: t \in \C \longmapsto (\Phi_t^g)^*X_f(x) \in \C^{2n}. \]
Clearly, the map $\Phi_t^g$ is analytic, and it sends $\mathcal{D}_{r_j-r',s_j-s'}(\omega_j)$ into $\mathcal{D}_{r_j,s_j}(\omega_j)$ for complex values of $t$ satisfying
\[ |t| < r'||X_g||_{r_j,s_j}^{-1}, \]
hence the function $F_x$ is well-defined and analytic on the disc $|t| < r'||X_g||_{r_j,s_j}^{-1}$. So applying the classical Cauchy estimate to each component of $F_x$ and optimizing with respect to $x \in \mathcal{D}_{r_j-r',s_j-s'}$ we obtain the desired inequality
\[ ||[X_f,X_g]||_{r_j-r',s_j-s'} < \frac{1}{r'}||X_g||_{r_j,s_j}||X_f||_{r_j,s_j}. \]
In case $g$ is integrable, the map $\Phi_t^g$ leaves invariant the action components, so the same reasoning can be applied on the larger disc
\[ |t| < s'||X_g||_{r_j,s_j}^{-1}, \]
giving the improved estimate
\[ ||[X_f,X_g]||_{r_j-r',s_j-s'} < \frac{1}{s'}||X_f||_{r_j,s_j}||X_g||_{r_j}. \]
\end{proof}

\subsection{Proof of Proposition~\ref{lemmeham}} 

Now we can pass to the proof of Proposition~\ref{lemmeham}. Given $\tilde{\varepsilon}>0$ which will be the size of our perturbating vector field $X_f$, let us introduce a slightly modified set of conditions $(\tilde{A}_j)$, for $j\in\{1,\dots,n\}$, where $(\tilde{A}_1)$ is
\begin{equation} 
\begin{cases}
mT_1\tilde{\varepsilon}\PM r_1, \; mT_1r_1\PM s_1, \; 0<r_1 \MP s_1, \\
\mathcal{B}_{r_1}(\omega_1)\neq \emptyset, 
\end{cases} \tag{$\tilde{A}_1$}
\end{equation} 
and for $j\in\{2,\dots,n\}$, ($\tilde{A}_j)$ is
\begin{equation} 
\begin{cases}
mT_j\tilde{\varepsilon}\PM r_j, \; mT_jr_j\PM s_j, \; 0<r_j \MP s_j,\\
\mathcal{B}_{r_j}(\omega_j)\neq \emptyset,\;\mathcal{D}_{r_j,s_j}(\omega_{j})\subseteq\mathcal{D}_{2r_{j-1}/3,2s_{j-1}/3}(\omega_{j-1}).  
\end{cases} \tag{$\tilde{A}_j$}
\end{equation}
These modifications take into account the fact that we will use the weighted norms $||\,.\,||_{r_j,s_j}$, for $j\in\{1,\dots,n\}$. 

\paraga The normal form lemma in terms of vector fields is the following.

\begin{proposition} \label{lemmevecti}
Consider $H=h+f$ on the domain $\mathcal{D}_{r_1,s_1}(\omega_1)$, with $||X_f||_{r_1,s_1} < \tilde{\varepsilon}$, and let $j\in\{1,\dots,n\}$. If $(\tilde{A}_i)$ is satisfied for any $i\in\{1,\dots,j\}$, then there exists an analytic symplectic transformation 
\[ \Psi_j: \mathcal{D}_{2r_j/3,2s_j/3}(\omega_j) \rightarrow \mathcal{D}_{r_1,s_1}(\omega_1)\] 
such that
\begin{equation*}
H \circ \Psi_j=h+g_j+f_j,
\end{equation*}
with $\{g_j,l_i\}=0$ for $i \in \{1, \dots, j\}$, and the estimates
\begin{equation*}
||X_{g_j}||_{2r_j/3,2s_j/3} \MP \tilde{\varepsilon}, \quad ||X_{f_j}||_{2r_j/3,2s_j/3} \MP e^{-m} \tilde{\varepsilon}. 
\end{equation*}
Moreover, we have $\Psi_j=\Phi_1 \circ \cdots \circ \Phi_j$ with
\[ \Phi_i: \mathcal{D}_{2r_i/3,2s_i/3}(\omega_i) \rightarrow \mathcal{D}_{r_i,s_i}(\omega_i)\] 
such that $|\Phi_i-\mathrm{Id}|_{2r_i/3,2s_i/3} \PM r_i$.
\end{proposition}

Let us see how this implies our Proposition~\ref{lemmeham}.

\begin{proof}[Proof of Proposition~\ref{lemmeham}.]
We know that $|f|_{r,s}<\varepsilon$, so we can apply Lemma~\ref{lemme1} with $r'=r_1$ and $s'=s_1$ to obtain
\[ |\partial_{I}f|_{r-r_1,s} < r_1^{-1}|f|_{r,s}, \quad |\partial_{\theta}f|_{r,s-s_1} < s_1^{-1}|f|_{r,s},\]
and hence
\[ ||X_f||_{r-r_1,s-s_1} < r_{1}^{-1}\varepsilon. \]
Now since $r_1 \PM r$ and $s_1 \PM s$ (this is part of assumption $(A_1)$), we have the inclusion $\mathcal{D}_{r_1,s_1}(\omega_1) \subseteq \mathcal{D}_{r-r_1,s-s_1}$ and hence
\[ ||X_f||_{r_1,s_1} \MP r_{1}^{-1}\varepsilon. \]
Set $\tilde{\varepsilon}=r_{1}^{-1}\varepsilon$, then for any $i\in\{1,\dots,j\}$, $(A_i)$ implies $(\tilde{A}_i)$ so that the Proposition~\ref{lemmevecti} can be applied: there exists an analytic symplectic transformation 
\[ \Psi_j: \mathcal{D}_{2r_j/3,2s_j/3}(\omega_j) \rightarrow \mathcal{D}_{r_1,s_1}(\omega_1)\] 
such that 
\begin{equation*}
H \circ \Psi_j=h+g_j+f_j,
\end{equation*}
with $\{g_j,l_i\}=0$ for $i \in \{1, \dots, j\}$, and the estimates
\begin{equation*}
||X_{g_j}||_{2r_j/3,2s_j/3} \MP \varepsilon r_{1}^{-1}, \quad ||X_{f_j}||_{2r_j/3,2s_j/3} \MP e^{-m} r_{1}^{-1}\varepsilon. 
\end{equation*}
Recalling the definition of our norm $||\,.\,||_{r_j,s_j}$, this readily implies 
\begin{equation*}
|\partial_\theta g_j|_{2r_j/3,2s_j/3}\MP \varepsilon s_1^{-1} \MP \varepsilon, \quad |\partial_\theta f_j|_{2r_j/3,2s_j/3}\MP e^{-m}\varepsilon s_1^{-1} \MP e^{-m}\varepsilon. 
\end{equation*}
Moreover, we have $\Psi_j=\Phi_1 \circ \cdots \circ \Phi_j$ with
\[ \Phi_i: \mathcal{D}_{2r_i/3,2s_i/3}(\omega_i) \rightarrow \mathcal{D}_{r_i,s_i}(\omega_i)\] 
such that $|\Phi_i-\mathrm{Id}|_{2r_i/3,2s_i/3} \PM r_i$.
\end{proof}

\paraga Hence it remains to prove Proposition~\ref{lemmevecti}. This will be done by induction on $j \in \{1,\dots,n\}$, and for that we shall need two iterative lemmas. The first iterative lemma is needed for the first step, that is to prove the statement for $j=1$, and it can be seen as an averaging process with respect to one fast angle. 

\begin{lemma}[First iterative lemma] \label{lemmeit}
Consider $H=h+g+f$ on the domain $\mathcal{D}_{r_1,s_1}(\omega_1)$, with $h$ integrable, $\{g,l_1\}=0$, and assume that
\begin{equation*}
||X_{g}||_{r_1,s_1} \MP \tilde{\varepsilon}, \quad ||X_{f}||_{r_1,s_1} < \tilde{\varepsilon}.
\end{equation*}
If we have 
\begin{equation*}
T_1\tilde{\varepsilon} < r' < s'
\end{equation*}
with two real numbers $r'$, $s'$ satisfying $0<r'<r_1$ and $0<s'<s_1$, then there exists an analytic symplectic transformation 
\[ \varphi_1: \mathcal{D}_{r_1-r',s_1-s'}(\omega_1) \rightarrow \mathcal{D}_{r_1,s_1}(\omega_1)\] 
such that $|\varphi_1-\mathrm{Id}|_{r_1-r',s_1-s'} < T_1\tilde{\varepsilon}$ and
\begin{equation*}
H \circ \varphi_1=h+g_++f_+,
\end{equation*}
with $\{g_+,l_1\}=0$ and the estimates
\begin{equation*}
||X_{g_+}||_{r_1,s_1} \MP \tilde{\varepsilon}, \quad||X_{g_+}-X_{g}||_{r_1,s_1} < \tilde{\varepsilon}, \quad ||X_{f_+}||_{r_1-r',s_1-s'} \MP \left(\frac{r_1}{s'}+\frac{\tilde{\varepsilon}}{r'}\right)T_1\tilde{\varepsilon}. 
\end{equation*}
\end{lemma}

\begin{proof}
We have $H=h+g+f$, with $h$ integrable, $g$ satisfying $\{g,l_1\}$ and $f$ a general term. Let us write
\[[f]_1=\frac{1}{T_1}\int_{0}^{T_1}f \circ \Phi_{t}^{l_1}dt,\]
the average of $f$ along the Hamiltonian flow of $l_1$.

Our transformation $\varphi_1=\Phi_{1}^{\chi}$ will be the time-one map of the Hamiltonian flow generated by some auxiliary function $\chi$ which satisfies 
\[ \{\chi,l_1\}=f-[f]_1. \]
The latter equation is easily solved by
\begin{equation} \label{chi}
\chi=\frac{1}{T_1}\int_{0}^{T_1}(f-[f]_1)\circ \Phi_{t}^{l_1}tdt,
\end{equation} 
and by Taylor formula, our transformed Hamiltonian writes
\[ H \circ \varphi_1=h+g_++f_+, \]
with 
\begin{equation*}
g_+=g+[f]_1, \quad f_+=\int_{0}^{1}\{h-l_1+g+f_t,\chi\}\circ \Phi_{t}^{\chi}dt, 
\end{equation*}
and $f_t=tf+(1-t)[f]_1$. By construction, $g_+$ still satisfies $\{g_+,l_1\}=0$, and 
\begin{equation*}
X_{g_+}-X_{g}=X_{[f]_1}=\frac{1}{T_1}\int_{0}^{T_1}(\Phi_{t}^{l_1})^* X_{f} dt.
\end{equation*}
Our hypothesis $||X_{f}||_{r_1,s_1} < \tilde{\varepsilon}$ immediately gives $||X_{g_+}-X_{g}||_{r_1,s_1} < \tilde{\varepsilon}$ and also $||X_{g_+}||_{r_1,s_1} \MP \tilde{\varepsilon}$. Similarly using~(\ref{chi}) we have the expression 
\begin{equation*}
X_{\chi}=\frac{1}{T_1}\int_{0}^{T_1}(\Phi_{t}^{l_1})^* X_{f-[f]_1}tdt,
\end{equation*} 
and hence $||X_{\chi}||_{r_1,s_1} < T_1\tilde{\varepsilon}$. By the hypothesis $T_1\tilde{\varepsilon} < r' < s'$ our transformation $\varphi_1$ maps $\mathcal{D}_{r_1-r',s_1-s'}(\omega_1)$ into $\mathcal{D}_{r_1,s_1}(\omega_1)$ and 
\[ |\varphi_1-\mathrm{Id}|_{r_1-r',s_1-s'} < T_1\tilde{\varepsilon}. \] 
Therefore it remains to estimate the vector field
\[ X_{f_+}=\int_{0}^{1}(\Phi_{t}^{\chi})^*[X_{h-l_1}+X_{g}+X_{f_t},X_{\chi}]dt, \]
and for that it is enough to estimate the brackets $[X_{f_t},X_{\chi}]$, $[X_{g},X_{\chi}]$ and $[X_{h-l_1},X_{\chi}]$. Using Lemma~(\ref{lemme3}), we find
\[ ||[X_{f_t},X_{\chi}]||_{r_1-r',s_1-s'} < \frac{1}{r'}||[X_{f_t}||_{r_1,s_1}||X_\chi||_{r_1,s_1} \MP \frac{\tilde{\varepsilon}}{r'}T_1\tilde{\varepsilon} \]
and
\[ ||[X_{g},X_{\chi}]||_{r_1-r',s_1-s'} < \frac{1}{r'}||[X_{g}||_{r_1,s_1}||X_\chi||_{r_1,s_1} \MP \frac{\tilde{\varepsilon}}{r'}T_1\tilde{\varepsilon}.\]
For the last bracket, note that $h-l_1$ is integrable so that we can use the improved estimate in Lemma~(\ref{lemme3}). By definition of the domain $\mathcal{D}_{r_1,s_1}(\omega_1)$, we have $||X_{h-l_1}||_{r_1} \MP r_1$ and hence
\[ ||[X_{h-l_1},X_{\chi}]||_{r_1-r',s_1-s'} < \frac{1}{s'}||X_{h-l_1}||_{r_1}||X_{\chi}||_{r_1,s_1} \MP \frac{r_1}{s'}T_1\tilde{\varepsilon}. \]
Putting the last three estimates together we arrive at
\[ ||X_{f_+}||_{r_1-r',s_1-s'} \MP \left(\frac{r_1}{s'}+\frac{\tilde{\varepsilon}}{r'}\right) T_1\tilde{\varepsilon}. \]
\end{proof}

Our second iterative lemma is needed for the inductive step, that is to go from $j$ to $j+1$. This is just a simple extension of the previous one. Let $j\in\{1,\dots,n-1\}$.

\begin{lemma}[Second iterative lemma] \label{lemmeit2}
Consider $H=h+g+f$ on the domain $\mathcal{D}_{r_{j+1},s_{j+1}}(\omega_{j+1})$, with $h$ integrable, $\{g,l_i\}=0$ for $i \in \{1,\dots,j+1\}$, $\{f,l_{i'}\}=0$ for $i'\in \{1,\dots,j\}$, and assume that
\begin{equation*}
||X_{g}||_{r_{j+1},s_{j+1}} \MP \tilde{\varepsilon}, \quad ||X_{f}||_{r_{j+1},s_{j+1}} \MP \tilde{\varepsilon}.
\end{equation*}
If we have 
\begin{equation*}
T_{j+1}\tilde{\varepsilon} \PM r' \PM s'
\end{equation*}
with two real numbers $r'$, $s'$ satisfying $0<r'<r_{j+1}$ and $0<s'<s_{j+1}$, then there exists an analytic symplectic transformation 
\[ \varphi_{j+1}: \mathcal{D}_{r_{j+1}-r',s_{j+1}-s'}(\omega_{j+1}) \rightarrow \mathcal{D}_{r_{j+1},s_{j+1}}(\omega_{j+1})\] 
such that $|\varphi_{j+1}-\mathrm{Id}|_{r_{j+1}-r',s_{j+1}-s'} \MP T_{j+1}\tilde{\varepsilon}$ and
\begin{equation*}
H \circ \varphi_{j+1}=h+g_++f_+,
\end{equation*}
with $\{g_+,l_i\}=0$ for $i \in \{1,\dots,j+1\}$, $\{f_+,l_{i'}\}=0$ for $i' \in \{1,\dots,j\}$, and the estimates
\[ ||X_{g_+}||_{r_{j+1},s_{j+1}} \MP \tilde{\varepsilon}, \quad ||X_{g_+}-X_{g}||_{r_{j+1},s_{j+1}} \MP \tilde{\varepsilon}, \] 
\[ ||X_{f_+}||_{r_{j+1}-r',s_{j+1}-s'} \MP \left(\frac{r_{j+1}}{s'}+\frac{\tilde{\varepsilon}}{r'}\right)T_{j+1}\tilde{\varepsilon}.\] 
\end{lemma}

\begin{proof}
Our Hamiltonian is $H=h+g+f$, $h$ is integrable and we have $\{g,l_i\}=0$ for $i \in \{1,\dots,j+1\}$ and $\{f,l_{i'}\}=0$ for $i' \in \{1,\dots,j\}$. Once again, our transformation $\varphi_{j+1}=\Phi_{1}^{\chi}$ will be the time-one map of the Hamiltonian flow generated by some auxiliary function $\chi$. 

We choose
\begin{equation} \label{chi2}
\chi=\frac{1}{T_{j+1}}\int_{0}^{T_{j+1}}(f-[f]_{j+1})\circ \Phi_{t}^{l_{j+1}}tdt, 
\end{equation}
where $[.]_{j+1}$ is the averaging along the Hamiltonian flow of $l_{j+1}$. Introducing the notation $f_t=tf+(1-t)[f]_{j+1}$, like in Lemma~\ref{lemmeit} we have
\[ H \circ \varphi_{j+1}=h+g_+ +f_+ \]
with
\[ g_+=g+[f]_{j+1}, \quad f_+=\int_{0}^{1}\{h-l_{j+1}+g+f_t,\chi\}\circ \Phi_{t}^{\chi}dt.  \]
We need to verify that we still have $\{g_+,l_i\}=0$ for $i \in \{1,\dots,j+1\}$ and $\{f_+,l_{i'}\}=0$ for $i' \in \{1,\dots,j\}$. By definition, $\{[f]_{j+1},l_{j+1}\}=0$, and for $i' \in \{1,\dots,j\}$, we compute
\begin{eqnarray*}
\{[f]_{j+1},l_{i'}\} & = & \frac{1}{T_{j+1}}\int_{0}^{T_{j+1}}\{f \circ \Phi_{t}^{l_{j+1}},l_{i'}\}dt \\
& = & \frac{1}{T_{j+1}}\int_{0}^{T_{j+1}}\{f \circ \Phi_{t}^{l_{j+1}},l_{i'}\circ \Phi_{t}^{l_{j+1}}\}dt \\
& = & \frac{1}{T_{j+1}}\int_{0}^{T_{j+1}}\{f,l_{i'}\}\circ \Phi_{t}^{l_{j+1}}dt \\
& = & 0.
\end{eqnarray*}
This proves that $\{g_+,l_i\}=\{g+[f]_{j+1},l_i\}=0$ for $i \in \{1,\dots,j+1\}$. Now a completely similar calculation shows that for $i' \in \{1,\dots,j\}$, $\{\chi,l_{i'}\}=0$, hence $l_{i'} \circ \Phi_{t}^{\chi}=l_{i'}$ and therefore
\[ \{f_+,l_{i'}\}=\int_{0}^{1}\{\{h-l_{j+1}+g+f_t,\chi\},l_{i'}\}\circ \Phi_{t}^{\chi}dt. \]
The double bracket in the expression above is zero, as a consequence of Jacobi identity and the fact that $\{h-l_{j+1}+g+f_t,l_{i'}\}=\{\chi,l_{i'}\}=0$, hence $\{f_+,l_{i'}\}=0$ for $i'\in\{1,\dots,j\}$.

To conclude, using our hypothesis $T_{j+1}\tilde{\varepsilon} \PM r' \PM s'$, as in Lemma~\ref{lemmeit} we can show that our transformation $\varphi_{j+1}$ maps $\mathcal{D}_{r_{j+1}-r',s_{j+1}-s'}(\omega_{j+1})$ into $\mathcal{D}_{r_{j+1},s_{j+1}}(\omega_{j+1})$ with $|\varphi_{j+1}-\mathrm{Id}|_{r_{j+1}-r',s_{j+1}-s'} \MP T_{j+1}\tilde{\varepsilon}$ and the estimates
\[ ||X_{g_+}||_{r_{j+1},s_{j+1}} \MP \tilde{\varepsilon}, \quad ||X_{g_+}-X_{g}||_{r_{j+1},s_{j+1}} \MP \tilde{\varepsilon}, \] 
\[ ||X_{f_+}||_{r_{j+1}-r',s_{j+1}-s'} \MP \left(\frac{r_{j+1}}{s'}+\frac{\tilde{\varepsilon}}{r'}\right)T_{j+1}\tilde{\varepsilon},\] 
are obtained in a completely analogous way.  
\end{proof}

\paraga We can eventually complete the proof of our normal form~\ref{lemmevecti}.

\begin{proof}[Proof of Proposition~\ref{lemmevecti}]
The proof is by induction on $j\in \{1,\dots,n\}$.

\bigskip

\textit{First step.} Here we assume $(\tilde{A}_1)$ and we will apply $m$ times our first iterative Lemma~\ref{lemmeit}, starting with the Hamiltonian
\[ H^0=H=h+g^0+f^0 \]
where $g^0=0$ and $f^0=f$ and choosing uniformly at each step
\[ r'=(3m)^{-1}r_1, \quad s'=(3m)^{-1}s_1. \]
Since $m \geq 1$, we have $0<r'<r_1$, $0<s'<s_1$ and using $(\tilde{A}_1)$, we have 
\[ T_1\tilde{\varepsilon} < r' < s', \]
so that the lemma can indeed be applied at each step. For $i \in \{0,\dots,m-1\}$, the Hamiltonian $H^i=h+g^i+f^i$ at step $i$ is transformed into
\[ H^{i+1}=H^i \circ \varphi_1^i=h+g^{i+1}+f^{i+1}. \]
For each $i \in \{0, \dots, m\}$, we obviously have $\{g^i,l_1\}=0$ and we claim that the estimates
\begin{equation} \label{claim}
||X_{g^i}||_{r_{1}^{i},s_{1}^{i}} \MP \tilde{\varepsilon}, \quad ||X_{f^i}||_{r_{1}^{i},s_{1}^{i}} < \tilde{\varepsilon}_i,  
\end{equation}
hold true, where we have set $\tilde{\varepsilon}_i=e^{-i}\tilde{\varepsilon}$, $r_{1}^{i}=r_1-ir'$ and $s_{1}^{i}=s_1-is'$. Assuming this claim, given $i \in \{0, \dots, m-1\}$, we have 
\[ \varphi_1^i: \mathcal{D}_{r_{1}^{i+1},s_{1}^{i+1}}(\omega_1) \longrightarrow  \mathcal{D}_{r_{1}^{i},s_{1}^{i}}(\omega_1),\]
so that $\Psi_1=\varphi_{1}^0 \circ \cdots \circ \varphi_{1}^{m-1}$ is well defined from $\mathcal{D}_{2r_1/3,2s_1/3}(\omega_1)$ to $\mathcal{D}_{r_1,s_1}(\omega_1)$. Setting $g_1=g^m$ and $f_1=f^m$, we finally obtain
\[ H \circ \Psi_1=h+g_1+f_1 \]
with the desired properties, that is $\{g_1,l_1\}=0$ and the estimates
\[ ||X_{g_1}||_{2r_1/3,2s_1/3} < \tilde{\varepsilon}, \quad ||X_{f_1}||_{2r_1/3,2s_1/3} < e^{-m} \tilde{\varepsilon}. \] 
Note that since $||X_{f^i}||_{r_{1}^{i},s_{1}^{i}}< \tilde{\varepsilon}_i$ for $i \in \{0, \dots, m-1\}$, we obtain 
\[ |\varphi_1^i-\mathrm{Id}|_{r_{1}^{i+1},s_{1}^{i+1}}< T_1 \tilde{\varepsilon}_i, \]
which gives
\[ |\Psi_1-\mathrm{Id}|_{2r_1/3,2s_1/3} \leq \sum_{k=0}^{m-1}T_1\tilde{\varepsilon}_k \MP T_1\tilde{\varepsilon}. \]
But recall that $mT_1\tilde{\varepsilon} \PM r_1$ and hence we can arrange
\[ |\Psi_1-\mathrm{Id}|_{2r_1/3,2s_1/3} \PM r_1. \]

Therefore to conclude the proof we need to establish the estimates~(\ref{claim}), and we may proceed by induction. For $i=0$, $g^0=0$ and $f^0=f$ so there is nothing to prove. Now assume that the estimates~(\ref{claim}) are satisfied for each $k \leq i$, where $i \in \{0, \dots, m-1\}$. For $k \in \{0, \dots, i\}$, since $||X_{f^k}||_{r_{1}^{k},s_{1}^{k}}< \tilde{\varepsilon}_k$ we get that
\[ ||X_{g^{k+1}}-X_{g^k}||_{r_{1}^{k+1},s_{1}^{k+1}} < \tilde{\varepsilon}_k, \]
and therefore
\[ ||X_{g^{i+1}}||_{r_{1}^{i+1},s_{1}^{i+1}} \leq \sum_{k=0}^{i} \tilde{\varepsilon}_k \MP \tilde{\varepsilon}, \]
so this gives the desired estimate for $X_{g^{i+1}}$. For $X_{f^{i+1}}$, note that
\[ ||X_{f^{i+1}}||_{r_{1}^{i+1},s_{1}^{i+1}} \MP T\left(\frac{r_1}{s'}+\frac{\tilde{\varepsilon}}{r'}\right)||X_{f^i}||_{r_{1}^{i},s_{1}^{i}}, \]
but
\[ T_1\left(\frac{r_1}{s'}+\frac{\tilde{\varepsilon}}{r'}\right) \EP \left(\frac{mT_1r_1}{s_1}+ \frac{mT_1\tilde{\varepsilon}}{r_1}\right),   \]
so choosing properly the implicit constants in $(\tilde{A}_1)$ we can ensure that
\[ T_1\left(\frac{r_1}{s'}+\frac{\tilde{\varepsilon}}{r'}\right) \PM \frac{1}{e} \]
which implies the estimate for $X_{f^{i+1}}$ and concludes this first step.

\bigskip

\textit{Inductive step.}
Now assume that the statement holds true for some $j\in\{1,\dots,n-1\}$, and we have to show that it remains true for $j+1$. By assumptions, there exists an analytic symplectic transformation 
\[ \Psi_j: \mathcal{D}_{2r_j/3,2s_j/3}(\omega_j) \rightarrow \mathcal{D}_{r_1,s_1}(\omega_1)\] 
such that 
\begin{equation*}
H \circ \Psi_j=h+g_j+f_j,
\end{equation*}
with $\{g_j,l_i\}=0$ for $i \in \{1, \dots, j\}$ and the estimates
\begin{equation*}
||X_{g_j}||_{2r_j/3,2s_j/3} \MP \tilde{\varepsilon}, \quad ||X_{f_j}||_{2r_j/3,2s_j/3} \MP e^{-m}\tilde{\varepsilon}.
\end{equation*}
Also, $\Psi_j=\Phi_1 \circ \cdots \circ \Phi_j$ with
\[ \Phi_i: \mathcal{D}_{2r_i/3,2s_i/3}(\omega_i) \rightarrow \mathcal{D}_{r_i,s_i}(\omega_i)\] 
such that $|\Phi_i-\mathrm{Id}|_{2r_i/3,2s_i/3} \PM r_i$ for $i\in\{1,\dots,n\}$. Furthermore, $(\tilde{A}_{j+1})$ holds. Now consider the Hamiltonian $h+g_j$, it is defined on $\mathcal{D}_{2r_j/3,2s_j/3}(\omega_j)$, hence by $(\tilde{A}_{j+1})$, it is also defined on the domain $\mathcal{D}_{r_{j+1},s_{j+1}}(\omega_{j+1})$ and it satisfies $\{g_j,l_i\}=0$ for $i \in \{1, \dots, j\}$. Moreover, we have the estimate
\[ ||X_{g_j}||_{r_{j+1},s_{j+1}} \leq ||X_{g_j}||_{2r_j/3,2s_j/3} \MP \tilde{\varepsilon}. \]
As in the first step, starting this time with the Hamiltonian
\[ h+g_j=h+g_j^0+f_j^0, \]
with $g_j^0=0$ and $f_j^0=g_j$, we can apply $m$ times our second iterative Lemma~\ref{lemmeit2} to have the following: there exists an analytic symplectic transformation 
\[ \Phi_{j+1}: \mathcal{D}_{2r_{j+1}/3,2s_{j+1}/3}(\omega_{j+1}) \rightarrow \mathcal{D}_{r_{j+1},s_{j+1}}(\omega_{j+1}) \] 
of the form $\Phi_{j+1}=\varphi_{j+1}^0 \circ\cdots\circ \varphi_{j+1}^{m-1}$ such that $|\Phi_{j+1}-\mathrm{Id}|_{2r_{j+1}/3,2s_{j+1}/3} \PM r_{j+1}$ and
\begin{equation*}
(h+g_j) \circ \Phi_{j+1}=h+g_j^m+f_j^m,
\end{equation*}
with $\{g_j^m,l_i\}=0$ for $i \in \{1,\dots,j+1\}$, and the estimates
\begin{equation*}
||X_{g_j^m}||_{2r_{j+1}/3,2s_{j+1}/3} \MP \tilde{\varepsilon}, \quad ||X_{f_j^m}||_{2r_{j+1}/3,2s_{j+1}/3} \MP e^{-m} \tilde{\varepsilon}. 
\end{equation*} 
Now we set 
\[ \Psi_{j+1}=\Psi_j \circ \Phi_{j+1}: \mathcal{D}_{2r_{j+1}/3,2s_{j+1}/3}(\omega_{j+1}) \rightarrow \mathcal{D}_{r_1,s_1}(\omega_1),\] 
which is well-defined by $(\tilde{A}_{j+1})$, to have
\begin{eqnarray*}
H \circ \Psi_{j+1} & = & (H \circ \Psi_j) \circ \Phi_{j+1} \\
& = & (h+g_j+f_j) \circ \Phi_{j+1} \\
& = & (h+g_j) \circ \Phi_{j+1} + f_j \circ \Phi_{j+1} \\
& = & h+g_j^m+f_j^m+f_j \circ \Phi_{j+1} \\
& = & h+g_{j+1}+f_{j+1}
\end{eqnarray*}
with $g_{j+1}=g_j^m$ and $f_{j+1}=f_j^m+f_j \circ \Phi_{j+1}$. The conclusions follow: $\{g_{j+1},l_i\}=0$ for $i \in \{1,\dots,j+1\}$, we have the estimate
\begin{equation*}
||X_{g_{j+1}}||_{2r_{j+1}/3,2s_{j+1}/3} \MP \tilde{\varepsilon},
\end{equation*}
and since
\[||X_{f_j \circ \Phi_{j+1}}||_{2r_{j+1}/3,2s_{j+1}/3} \leq ||X_{f_j}||_{r_{j+1},s_{j+1}} \leq ||X_{f_j}||_{2r_{j}/3,2s_{j}/3}\]
we also have
\begin{eqnarray*}
||X_{f_{j+1}}||_{2r_{j+1}/3,2s_{j+1}/3} & \leq & ||X_{f_j^m}||_{2r_{j+1}/3,2s_{j+1}/3} + ||X_{f_j}||_{2r_j/3,2s_j/3} \\
& \MP & e^{-m} \tilde{\varepsilon}.
\end{eqnarray*}
The proof is therefore complete.
\end{proof}

\section{SDM functions} \label{SDM}

In this appendix, we will study our class of SDM functions. We will first show in~\ref{b1} that they satisfy an adapted steepness property, which we used in the proof of our exponential estimates, and then in~\ref{b2} we will prove that they are generic. These results are similar to~\cite{Nie07}. 

\subsection{Steepness.} \label{b1}

\paraga We denote by $GA_B(n,k)$ the set of all affine subspaces of $\R^n$ of dimension $k$ intersecting the ball $B$, and by $GA_{B}^{L}(n,k)$ those subspaces with direction in $G^{L}(n,k)$ (the latter is the space of linear subspaces of $\R^n$ of dimension $k$ whose orthogonal complement is spanned by integer vectors of length less than or equal to $L$). Let us recall the classical steepness condition, originally introduced by N.N. Nekhoroshev (\cite{Nek77}). 

\begin{definition}
A function $h \in C^2(B)$ is said to be steep if it has no critical points and if for any $k \in \{1, \dots, n-1\}$, there exist an index $p_k>0$ and coefficients $C_k>0$, $\delta_k>0$ such that for any affine subspace $\lambda_k \in GA_B(n,k)$ and any continuous curve $\Gamma : [0,1] \rightarrow \lambda_k \cap B$ with 
\[ \Vert\Gamma(0)-\Gamma(1)\Vert = r <\delta_k,\] 
there exists $t_*\in [0,1]$ such that:
\begin{equation*}
\begin{cases}
\Vert\Gamma (t)-\Gamma (0)\Vert < r, \quad t\in [0,t_*], \\
\left\Vert\Pi_{\Lambda_k}(\nabla h(\Gamma (t_*)))\right\Vert > C_k r^{p_k}
\end{cases}
\end{equation*}
where $\Pi_{\Lambda_k}$ is the projection onto $\Lambda_k$, the direction of $\lambda_k$.

The function is said to be symmetrically steep (or shortly S-steep) if the above property is also satisfied for $k=n$, with an index $p_n>0$ and coefficients $C_n>0$, $\delta_n>0$.
\end{definition}

Let us remark that S-steep functions are allowed to have critical points. Those definitions are rather obscure, but in fact it can be given a simpler and more geometric interpretation, as was shown by Ilyashenko (\cite{Ily86}) and Niederman (\cite{Nie06}). Important examples of steep functions are given by the class of strictly convex (or quasi-convex) functions, with all the steepness indices equal to one. 

\paraga A typical example of non-steep function, which is due to Nekhoroshev, is $h(I_1,I_2)= I_{1}^{2}-I_{2}^{2}$, and it is not exponentially stable: for the perturbation $h_\varepsilon(I_1,I_2)= I_{1}^{2}-I_{2}^{2}+ \varepsilon \sin(I_1+I_2)$, any solution with $I_1(0)=I_2(0)$ has a fast drift, that is a drift of order one on a time scale of order $\varepsilon^{-1}$ (this is obviously the fastest drift possible). But adding a third order term in the previous example (for example $I_{2}^{3}$) we recover steepness, and this is in fact a general phenomenon. Indeed, non-steep functions has infinite codimension among smooth functions, or more precisely, if $J_r(n)$ is the space of $r$-jets of $C^{\infty}$ functions on an open set of $\R^n$, then Nekhoroshev proved in \cite{Nek79} that the set of $r$-jets of non-steep functions is an algebraic subset of $J_r(n)$ which codimension goes to infinity has $r$ goes to infinity. In this sense, steep functions are ``generic". However, for $n\geq 3$, a quadratic Hamiltonian is steep only if it is sign definite, which is a strong assumption, and more generally a polynomial is generically steep only if its degree is sufficiently high (of order $n^2$ if $n$ is the number of degrees of freedom). Hence polynomials of lower degree are generically non-steep (see \cite{LM88}). This is clearly a shortcoming, and we will see at the end of the next section the advantage of our genericity condition. 

\paraga Steepness (or S-steepness) is a sufficient condition to ensure exponential stability, but this is not necessary, as was first noticed by Morbidelli and Guzzo (see \cite{MG96}). They considered the Hamiltonian $h(I_1,I_2)= I_{1}^{2}-\alpha I_{2}^{2}$, which is non-steep for any value of $\alpha>0$, and noticed that a ``fast drift" is not possible if $\sqrt\alpha$ is ``strongly" irrational. Therefore a Diophantine condition on $\sqrt\alpha$ should ensure exponential stability. 

Such considerations were then generalized by Niederman who introduced the class of ``Diophantine Morse" functions and who proved that they are exponentially stable (\cite{Nie07}). The only difference between these functions and the ``Simultaneous Diophantine Morse" functions we use in this paper is that Diophantine Morse functions consider subspaces in $G_L(n,k)$, which are generated by integer vectors of length bounded by $L$, while here we are looking at subspaces in $G^L(n,k)$ where the latter condition is imposed on the orthogonal complement. This reflects the difference between the method of proof: in (\cite{Nie07}) the analytic part was based on classical small divisors techniques (that is linear Diophantine approximation) and therefore required an adapted geometric assumption, while here we simply rely on the most basic theorem of simultaneous Diophantine approximation (and this explains the name Simultaneous Diophantine Morse functions).  

\paraga In both cases, the use of such a class of functions has two advantages. The first one is that these functions are generic in a much more clearer sense than steep functions, and this will be explained in the next section. The second advantage is that they are in some sense more general than the usual steep functions, since we only have to consider curves in some specific affine subspaces. This is explained in the proposition below. 

\begin{proposition}\label{steep}
Let $h \in SDM_{\gamma}^{\tau}(B)$, assume that $|h|_{C^3(B)}<M$ and take $r<1$. Then for any affine subspace $\lambda \in GA_{B}^{L}(n,k)$ and any continuous curve $\Gamma : [0,1] \rightarrow \lambda \cap B$ with 
\[\Vert\Gamma(0)-\Gamma(1)\Vert = r <(2M)^{-1}\gamma L^{-\tau},\] 
there exists $t_*\in [0,1]$ such that:
\begin{equation*}
\begin{cases}
\Vert\Gamma (t)-\Gamma (0)\Vert \leq r, \quad t\in [0,t_*], \\
\left\Vert\Pi_{{\Lambda}}(\nabla h(\Gamma (t_*)))\right\Vert > \demi r^2
\end{cases}
\end{equation*}
where $\Pi_\Lambda$ is the projection onto $\Lambda$, the direction of $\lambda$.
\end{proposition}

\begin{proof}
It is enough to check that these properties are satisfied for a vector space $\Lambda\in G^{L}(n,k)$, since any affine subspace $\lambda \in GA_{B}^{L}(n,k)$ is of the form $\lambda=v+\Lambda$ with $\Lambda\in G^{L}(n,k)$ for some vector $v$. So consider a continuous curve $\Gamma : [0,1] \rightarrow \Lambda \cap B$ with length $r<1$ satisfying
\[ \Vert\Gamma(0)-\Gamma(1)\Vert = r <(2M)^{-1}\gamma L^{-\tau}. \] 
We will denote by $(\alpha(t),\beta)$ the coordinates of $\Gamma(t)$ for $t\in [0,1]$ in a basis adapted to the orthogonal decomposition $\Lambda \oplus \Lambda^\perp$. Therefore
\[ \Vert\Pi_\Lambda(\nabla h(\Gamma(t)))\Vert=\Vert\partial_\alpha h_\Lambda(\alpha(t),\beta)\Vert\] 
for all $t\in[0,1]$. We will distinguish distinguish two cases.

For the first one, we suppose that 
\[ \Vert\partial_\alpha h_\Lambda(\alpha(0),\beta)\Vert > 2^{-1} r^2, \]
so the conclusion trivially holds for $t_*=0$.

For the second one, we have 
\begin{equation}\label{jojo2}
\Vert\partial_\alpha h_\Lambda(\alpha(0),\beta)\Vert \leq 2^{-1} r^2,
\end{equation}
but since $r^2 < r < \gamma L^{-\tau}$, this gives 
\begin{equation*}
\Vert\partial_\alpha h_\Lambda(\alpha(0),\beta)\Vert \leq \gamma L^{-\tau}. 
\end{equation*}
Now $h \in SDM_{\gamma}^{\tau}(B)$, so we can apply the definition at the point $(\alpha(0),\beta)$, and for any $\eta \in \R^k\setminus\{0\}$ we obtain
\begin{equation}\label{jojo}
\Vert\partial_{\alpha\alpha} h_\Lambda(\alpha(0),\beta).\eta\Vert>\gamma L^{-\tau}\Vert\eta\Vert. 
\end{equation}
Take any $\tilde{\alpha}$ such that $\Vert\tilde{\alpha}-\alpha(0)\Vert<(2M)^{-1}\gamma L^{-\tau}$. We can apply Taylor formula with integral remainder to obtain
\[ \partial_\alpha h_\Lambda(\tilde{\alpha},\beta)-\partial_\alpha h_\Lambda(\alpha(0),\beta)=\int_{0}^{1}\partial_{\alpha\alpha} h_\Lambda(\alpha(0)+t(\tilde{\alpha}-\alpha(0)),\beta).(\tilde{\alpha}-\alpha(0))dt. \]
Now since $M$ bounds the third derivative of $h$, we have
\[ \Vert\partial_{\alpha\alpha} h_\Lambda(\alpha(0)+t(\tilde{\alpha}-\alpha(0)),\beta)-\partial_{\alpha\alpha} h_\Lambda(\alpha(0),\beta)\Vert \leq Mt\Vert\tilde{\alpha}-\alpha(0)\Vert \leq 2^{-1}\gamma L^{-\tau}t, \]
and this yields
\begin{eqnarray*}
\Vert\partial_\alpha h_\Lambda(\tilde{\alpha},\beta)-\partial_\alpha h_\Lambda(\alpha(0),\beta)\Vert & \geq & \Vert\partial_{\alpha\alpha} h_\Lambda(\alpha(0),\beta).(\tilde{\alpha}-\alpha(0))\Vert \\
& & -2^{-1}\gamma L^{-\tau}\int_{0}^{1}t\Vert\tilde{\alpha}-\alpha(0)\Vert dt, 
\end{eqnarray*}
which in turns, using~(\ref{jojo}) with $\eta=\tilde{\alpha}-\alpha(0)$, gives
\begin{eqnarray} \label{jojo3}
\Vert\partial_\alpha h_\Lambda(\tilde{\alpha},\beta)-\partial_\alpha h_\Lambda(\alpha(0),\beta)\Vert & \geq & \left(\gamma L^{-\tau}-2^{-1}\gamma L^{-\tau}\int_{0}^{1}tdt \right)\Vert \tilde{\alpha}-\alpha(0)\Vert \nonumber \\
& \geq & 2^{-1}\gamma L^{-\tau}\Vert\tilde{\alpha}-\alpha(0)\Vert. 
\end{eqnarray}
Now we define
\[ t_*=\inf_{t\in [0,1]}\{ \Vert\Gamma(t)-\Gamma(0)\Vert= r \}, \]
so trivially we have
\[ \Vert\Gamma(t)-\Gamma(0)\Vert \leq r, \quad t\in[0,t_*]. \]
Furthermore, we have
\[ \Vert\partial_\alpha h_\Lambda(\alpha(t_*),\beta)\Vert \geq \Vert\partial_\alpha h_\Lambda(\alpha(t_*),\beta) - \partial_\alpha h_\Lambda(\alpha(0),\beta)\Vert - \Vert\partial_\alpha h_\Lambda(\alpha(0),\beta)\Vert ,\]
and so using~(\ref{jojo2}), (\ref{jojo3}) and recalling that $\Vert\alpha(t_*)-\alpha(0)\Vert=r$ and $\gamma L^{-\tau}>2r$ we obtain
\begin{eqnarray*}
\Vert\partial_\alpha h_\Lambda(\alpha(t_*),\beta)\Vert & \geq & 2^{-1}\gamma L^{-\tau} r - 2^{-1} r^2   \\
& > & r^2 - 2^{-1} r^2 \\
& = & 2^{-1} r^2,
\end{eqnarray*}
and this is the desired estimate.
\end{proof} 

\subsection{Prevalence} \label{b2}

\paraga Here we will prove our results of genericity concerning SDM functions, that is Theorem~\ref{thmpreva} and Corollary~\ref{corpreva}. Our main tool is the following lemma, which is proved in \cite{Nie07} and relies on the quantitative Morse-Sard theory developed by Yomdin (see \cite{YC04} and \cite{Yom83}). Let us denote by $\lambda_k$ the $k$-dimensional Lebesgue measure.  

\begin{lemma} \label{MS}
Let $g \in C^{2n+1}(B,\R^k)$. Then for any $\kappa \in ]0,1[$ there exist a subset $\mathcal{C}_\kappa \subseteq \R^k$ with 
\[ \lambda_k(\mathcal{C}_\kappa) \leq c_k\sqrt{\kappa},\] 
where $c_k$ only depends on $k$, such that for any $\zeta \notin \mathcal{C}_\kappa$, the function $g^\zeta$ defined by $g^\zeta(x)=g(x)-\zeta$ satisfies the following: for any $x\in B$, 
\[ \Vert g^\zeta(x) \Vert \leq \kappa \Longrightarrow  \Vert dg^\zeta(x).\nu\Vert>\kappa \Vert\nu \Vert, \]
for any $\nu \in \R^n\setminus\{0\}$.  
\end{lemma} 

In the above statement, the set $\mathcal{C}_\kappa$ is a ``nearly-critical set" for the function $g$. 

\paraga Let us prove Theorem~\ref{thmpreva}. 

\begin{proof}[Proof of Theorem~\ref{thmpreva}]
Recall that we are given a function $h \in C^{2n+2}(B)$. The proof is divided in two steps: first, we will describe the set of parameters $\xi\in\R^n$ for which the function $h_\xi$, defined by $h_\xi(I)=h(I)-\xi.I$, is not in $SDM^\tau(B)$, and then, in a second step, we will show that this set has zero Lebesgue measure, for $\tau>2(n^2+1)$. In the sequel, given $k\in\{1,\dots,n\}$, we denote by $\lambda_k$ the Lebesgue measure of $\R^k$. 

\bigskip

\textit{First step.} Given an element $\Lambda \in G^L(n,k)$, let $\Pi_\Lambda$ the projection onto this subspace and consider the associate function $h_\Lambda$ (recall that $h_\Lambda$ is just the function $h$ written in coordinates adapted to the orthogonal decomposition $\Lambda\oplus\Lambda^{\perp}$). Let us define the function
\[ g=\partial_\alpha h_\Lambda, \]
which belongs $C^{2n+1}(B,\R^k)$, and apply to this function Lemma~\ref{MS} with the value $\kappa=\gamma L^{-\tau}$. We find a ``nearly-critical" set $\mathcal{C}_\kappa=\mathcal{C}_{\gamma,\tau,L} \subseteq \R^k$ with the measure estimate
\begin{equation}\label{vol1}
\lambda_k(\mathcal{C}_{\gamma,\tau,L}) \leq c_k \gamma^{\frac{1}{2}}L^{-\frac{\tau}{2}},
\end{equation}
such that for any $\zeta \notin \mathcal{C}_{\gamma,\tau,L}$ and any $(\alpha,\beta) \in B$,
\begin{equation}\label{la1} 
 \Vert g^\zeta(\alpha,\beta) \Vert \leq \kappa \Longrightarrow  \Vert dg^\zeta(\alpha,\beta).\nu \Vert>\kappa \Vert\nu \Vert, 
\end{equation}
for any $\nu\in\R^n\setminus\{0\}$.

Now take any $\zeta \notin \mathcal{C}_{\gamma,\tau,L}$, any $\xi \in \Pi_{\Lambda}^{-1}(\zeta)$ and consider the modified function $h_\xi$ as well as its version $h_{\xi,\Lambda}$. Since
\[ \partial_\alpha h_{\xi,\Lambda}=\partial_\alpha h_\Lambda - \zeta=g-\zeta=g^\zeta, \]
and $\partial_{\alpha,\alpha} h_{\xi,\Lambda}=\partial_{\alpha,\alpha} h_{\Lambda}$ is just some restriction of $dg$, the estimate~(\ref{la1}) gives for any $(\alpha,\beta)\in B$, 
\begin{equation}\label{la2}
\Vert\partial_\alpha h_{\xi,\Lambda}(\alpha,\beta)\Vert\leq \gamma L^{-\tau} \Longrightarrow \Vert \partial_{\alpha,\alpha} h_{\xi,\Lambda}(\alpha,\beta).\eta\Vert>\gamma L^{-\tau}\Vert\eta\Vert 
\end{equation}  
for any $\eta\in\R^n\setminus\{0\}$. So let $\mathcal{C}_{\gamma,\tau,L,\Lambda}=\Pi_{\Lambda}^{-1}(\mathcal{C}_{\gamma,\tau,L})$, and define
\[ \mathcal{C}_{\gamma,\tau}=\bigcup_{L \in \N^*}\bigcup_{k\in\{1,\dots,n\}}\bigcup_{\Lambda \in G^L(n,k)}\mathcal{C}_{\gamma,\tau,L,\Lambda}. \]
As a consequence of the estimate~(\ref{la2}), the function $h_\xi \in SDM_{\gamma}^{\tau}(B)$ provided that $\xi \notin \mathcal{C}_{\gamma,\tau}$, hence $h_\xi \in SDM^{\tau}(B)$ provided that $\xi \notin \mathcal{C}_{\tau}$, where
\[ \mathcal{C}_{\tau}=\bigcap_{\gamma>0}\mathcal{C}_{\gamma,\tau}. \]

\bigskip

\textit{Second step.} It remains to prove that $\mathcal{C}_{\tau}$ has zero Lebesgue measure under our assumption that $\tau>2(n^2+1)$. For an integer $m\in\N^{*}$, we define $\mathcal{C}^{m}_{\gamma,\tau,L,\Lambda}$ (resp. $\mathcal{C}^{m}_{\gamma,\tau}$ and $\mathcal{C}^{m}_{\tau}$) as the intersection of $\mathcal{C}_{\gamma,\tau,L,\Lambda}$ (resp. $\mathcal{C}_{\gamma,\tau}$ and $\mathcal{C}_{\tau}$) with the ball of $\R^n$ of radius $m$ centered at the origin. As a consequence of~(\ref{vol1}) and Fubini-Tonelli theorem, one has
\begin{equation*}
\lambda_n(\mathcal{C}^{m}_{\gamma,\tau,L,\Lambda}) \leq V_{n,m} c_k \gamma^{\frac{1}{2}}L^{-\frac{\tau}{2}}
\end{equation*}
where $V_{n,m}=m^n \pi^{n/2}\Gamma(n/2+1)^{-1}$ is the volume of the ball of $\R^n$ of radius $m$ centered at the origin. Therefore
\begin{equation*}
\lambda_n\left(\bigcup_{\Lambda \in G^L(n,k)}\mathcal{C}^{m}_{\gamma,\tau,L,\Lambda}\right) \leq |G^L(n,k)|V_{n,m} c_k L^{-\frac{\tau}{2}} \gamma^{\frac{1}{2}},
\end{equation*}
with $|G^L(n,k)|$ the cardinal of $G^L(n,k)$. But obviously $|G^L(n,k)|\leq L^{n^2}$ and hence
\begin{equation*}
\lambda_n\left(\bigcup_{\Lambda \in G^L(n,k)}\mathcal{C}^{m}_{\gamma,\tau,L,\Lambda}\right) \leq V_{n,m} c_k L^{n^2-\frac{\tau}{2}} \gamma^{\frac{1}{2}}.
\end{equation*}
Now
\begin{equation*}
\lambda_n\left(\bigcup_{k\in\{1,\dots,n\}}\bigcup_{\Lambda \in G^L(n,k)}\mathcal{C}^{m}_{\gamma,\tau,L,\Lambda}\right) \leq V_{n,m} \left(\sum_{k=1}^{n}c_k\right)L^{n^2-\frac{\tau}{2}}\gamma^{\frac{1}{2}},
\end{equation*}
and so
\begin{equation*}
\lambda_n(\mathcal{C}^{m}_{\gamma,\tau})\leq V_{n,m} \left(\sum_{k=1}^{n}c_k \right)\left(\sum_{L=1}^{+\infty}L^{n^2-\frac{\tau}{2}}\right)\gamma^{\frac{1}{2}}
\end{equation*}
where the sum in the right-hand side of the last estimate is finite since we are assuming $\tau>2(n^2+1)$. This shows that
\[ \lambda_n(\mathcal{C}^{m}_{\tau})=\inf_{\gamma >0}\lambda_n(\mathcal{C}^{m}_{\gamma,\tau})=0, \] 
and as $\mathcal{C}_{\tau}=\bigcup_{m\geq 1}\mathcal{C}^{m}_{\tau}$ we finally obtain
\[ \lambda_n(\mathcal{C}_{\tau})=0, \]
and this concludes the proof.
\end{proof}

\paraga As we mentioned in the introduction, there is a notion of genericity in infinite dimensional vector spaces called prevalence, first introduced in a different setting by Christensen (\cite{Chr73}) and rediscovered by Hunt, Sauer and Yorke (\cite{HSY92}, see also \cite{OY05} and \cite{HK10}). 

\begin{definition}
Let $E$ be a completely metrizable topological vector space. A Borel subset $S \subseteq E$ is said to be shy if there exists a Borel measure $\mu$ on $E$, with $0<\mu(C)<\infty$ for some compact set $C\subseteq E$, and such that $\mu(x+S)=0$ for all $x\in E$. 

An arbitrary set is called shy if it is contained in a shy Borel subset, and finally the complement of a shy set is called prevalent.
\end{definition}  

The following ``genericity" properties are easy to check (\cite{OY05}, \cite{HK10}): a prevalent set is dense, a set containing a prevalent set is also prevalent, and prevalent sets are stable under translation and countable intersection. 

Furthermore, we have an easy but useful criterion for a set to be prevalent.

\begin{proposition}[\cite{HSY92}]\label{proppreva}
Let $P$ be a subset of $E$. Suppose there exists a finite-dimensional subspace $F$ of $E$ such that $x+P$ has full $\lambda_{F}$-measure for all $x\in E$. Then $P$ is prevalent.
\end{proposition}

\paraga Now we can prove our Corollary~\ref{corpreva}. 

\begin{proof}[Proof of Corollary~\ref{corpreva}]
Let $E=C^{2n+2}(B)$, $P=SDM^\tau(B)$ for $\tau>2(n^2+1)$ and $F$ the space of linear forms of $\R^n$ restricted to $B$. Then $F$ is a linear subspace of $C^{2n+2}(B)$ of dimension $n$, and the conclusion follows immediately from Theorem~\ref{thmpreva} and the above Proposition~\ref{proppreva}.
\end{proof}

\paraga To conclude, let us compare our generic condition with the usual steepness property. First, our condition is prevalent in the space $C^{k}(B)$, with $k\geq 2n+2$, and this is not true for steep functions. But more importantly, as prevalence is nothing but ``full Lebesgue measure" in finite dimension, given any non zero integers $m$ and $n$, Lebesgue almost all polynomial Hamiltonian $h_m$ of degree $m$ with $n$ degrees of freedom is $SDM$, but not steep unless $m$ is of order $n^2$. This remark turns out to be very useful when studying the stability of invariant tori under generic conditions (see \cite{Bou09}).

\bigskip

{\it Acknowledgements.}

The authors wish to thank Jacques F\'ejoz and Jean-Pierre Marco for useful discussions, and A.B. also thanks the ASD team at Observatoire de Paris for its hospitality, especially Alain Albouy and Alain Chenciner for their kind support. Both authors thank the CRM of Barcelona where this work was initiated during the semester ``Stability and Instability in Mechanical Systems". Finally, the authors wish to thank the referee for his careful reading.

\addcontentsline{toc}{section}{References}
\bibliographystyle{amsalpha}
\bibliography{NwsdFinal}
\end{document}